\documentclass[reqno,a4paper,11pt]{amsart}
\usepackage[T1]{fontenc}
\usepackage[utf8]{inputenc}
\usepackage{lmodern}
\usepackage[english]{babel}

\usepackage[%
    a4paper,
	left=2.5cm,       
	right=2.5cm,      
	top=3.5cm,        
	bottom=3.5cm,     
	heightrounded,    
	bindingoffset=0mm 
]{geometry}

\usepackage{amsmath}
\usepackage{amssymb}
\usepackage{amsfonts}
\usepackage{mathtools}
\usepackage{mathrsfs}
\usepackage{enumerate}
\usepackage{xcolor}
\usepackage[normalem]{ulem}

\usepackage[]{hyperref}
\hypersetup{
    colorlinks=true,       
    linkcolor=blue,          
    citecolor=red,        
    filecolor=red,      
    urlcolor=cyan           
}

\usepackage[noabbrev,capitalize]{cleveref}

\usepackage{bookmark}
\usepackage[initials]{amsrefs}

\theoremstyle{plain}
\newtheorem{theorem}{Theorem}[section]
\newtheorem{lemma}[theorem]{Lemma}
\newtheorem{prop}[theorem]{Proposition}

\newtheorem{corollary}[theorem]{Corollary}

\theoremstyle{definition}
\newtheorem{definition}[theorem]{Definition}

\theoremstyle{remark}
\newtheorem{remark}[theorem]{Remark}

\numberwithin{equation}{section}

\newcommand{\eps}{\varepsilon}

\newcommand{\R}{\mathbb R}
\newcommand{\N}{\mathbb N}
\newcommand{\T}{\mathbb T}

\newcommand{\jump}{\medskip}

\DeclareMathOperator{\diver}{div}
\DeclareMathOperator{\loglip}{Lip_{log}}
\DeclareMathOperator{\Lip}{Lip}

\title[Full double H\"older regularity of the pressure]{Full double H\"older regularity of the 
 pressure \\ in bounded domains}

\author[L.~De Rosa]{Luigi De Rosa}
\address[L.~De Rosa]{Department Mathematik und Informatik, Universität Basel, Spiegelgasse~1, 4051 Basel, Switzerland}
\email{luigi.derosa@unibas.ch}

\author[M.~Latocca]{Micka\"el Latocca}
\address[M.~Latocca]{Department of Mathematics, University of Maryland, College Park, MD 20742, USA }
\email{mickael.latocca@umd.edu}

\author[G.~Stefani]{Giorgio Stefani}
\address[G.~Stefani]{Scuola Internazionale Superiore di Studi Avanzati (SISSA), via Bonomea~265, 34136 Trieste (TS), Italy}
\email{gstefani@sissa.it {\normalfont\it or} giorgio.stefani.math@gmail.com}


\date{\today}


\thanks{
\textit{Acknowledgements}.
The first-named author has been partially funded by the SNF grant FLUTURA: Fluids, Turbulence, Advection No. 212573,  and he is indebted to Jaemin Park for fruitful conversations about weak Euler solutions on bounded domains. The second-named author thanks Nicolas Burq for an enlightening discussion leading to the choice of geodesic normal coordinates. 
The third-named author is member of the Istituto Nazionale di Alta Matematica (INdAM), Gruppo Nazionale per l'Analisi Matematica, la Probabilità e le loro Applicazioni (GNAMPA), is partially supported by the INdAM--GNAMPA 2022 Project \textit{Analisi geometrica in strutture subriemanniane}, codice CUP\_E55\-F22\-00\-02\-70\-001,  by the INdAM--GNAMPA 2023 Project \textit{Problemi variazionali per funzionali e operatori non-locali}, codice CUP\_E53\-C22\-00\-19\-30\-001, and has received funding from the European Research Council (ERC) under the European Union’s Horizon 2020 research and innovation program (grant agreement No.~945655).
}

\allowdisplaybreaks
\usepackage{graphicx}

\begin{document}

\begin{abstract}
 We consider H\"older continuous weak solutions $u\in C^\gamma(\Omega)$, $u\cdot n|_{\partial \Omega}=0$, of the incompressible Euler equations on a bounded and simply connected domain $\Omega\subset\R^d$. If $\Omega$ is of class $C^{2,1}$ then the corresponding pressure satisfies $p\in C^{2\gamma}_*(\Omega)$ in the case $\gamma\in (0,\frac{1}{2}]$, where $C^{2\gamma}_*$ is the H\"older--Zygmund space, which coincides with the usual H\"older space for $\gamma<\frac12$. This result, together with our previous one in \cite{DLS22} covering the case $\gamma\in(\frac12,1)$, yields the full double regularity of the pressure on bounded and sufficiently regular domains. The interior regularity comes from the corresponding $C^{2\gamma}_*$ estimate for the pressure on the whole space $\R^d$, which in particular extends and improves the known double regularity results (in the absence of a boundary) in the borderline case $\gamma=\frac{1}{2}$. The boundary regularity features the use of local normal geodesic coordinates, pseudodifferential calculus and a fine Littlewood--Paley analysis of the modified equation in the new coordinate system.  We also discuss the relation between different notions of weak solutions,  a step which plays a major role in our approach. 
\end{abstract}

\maketitle

\section{Introduction}

Let $d\geq 2$ and let $\Omega\subset\R^d$ be a bounded and simply connected domain of class $C^2$.  The time evolution in $\Omega$ of an incompressible inviscid fluid is described by the \emph{Euler equations}
\begin{equation}\label{E}
\left\{\begin{array}{rcll}
\partial_t u+ \diver (u \otimes u) +\nabla p &=& 0 &\text{in } \Omega\times (0,T)\\[1mm]
\diver u& = &0&\text{in } \Omega\times (0,T)\\[1mm]
u\cdot n&=&0&\text{on } \partial \Omega\times (0,T),
\end{array}\right.
\end{equation}
where $u\colon\Omega\times(0,T)\to\R^d$ and $p\colon\Omega\times(0,T)\to\R$ are the \emph{velocity} of the fluid and its \emph{hydrodynamic pressure}, respectively, and $n\colon\partial\Omega\to\R^d$ is the outward unit normal to $\partial \Omega$.  The boundary condition $u(\cdot, t)\cdot n=0$ on $\partial \Omega$ is the usual \emph{no-flow condition}, which prohibits the fluid to cross the boundary of the container $\Omega$.

\subsection{The pressure equation}
Forgetting about the time dependence of the unknown and only focusing on the spatial one, straightforward computations yield that the pressure $p$ solves the following elliptic Neumann boundary value problem
\begin{equation}
    \label{eq.pressureEq}
    \left\{
    \begin{array}{rcll}
        -\Delta p &=& \diver \diver (u\otimes u) & \text{ in }\Omega\\[2mm]
        \partial_n p&=&u\otimes u : \nabla n & \text{ on } \partial \Omega. 
    \end{array}
    \right.
\end{equation}
Indeed, the interior elliptic equation is obtained by taking the divergence of the first equation in~\eqref{E}, while  the boundary condition  follows by scalar multiplying the same equation by the unit normal $n\colon\partial\Omega\to\R^d$ and noticing that 
\begin{equation}
\label{reassembl_bound_cond}
\begin{split}
\partial_n p&=\nabla p\cdot n
=-\diver (u\otimes u)\cdot n =-\partial_i(u^i u^j)n^j
=-u^i\partial_i(u^j ) n^j
\\
&=-u_i\partial_i(u^j n^j)+u^iu^j\partial_i n^j
=-u\cdot\nabla(u\cdot n)+u\otimes u:\nabla n
=u\otimes u:\nabla n.
\end{split}
\end{equation}
Here we  used that $\partial \Omega$ is a level set of the scalar function $u\cdot n$, so $\nabla (u\cdot n)|_{\partial \Omega}$ is parallel to $n$.  In the previous chain of equalities, we implicitly assumed that the normal, which always satisfies $n\in C^1(\partial \Omega)$ on any $C^2$ domain, is extended to a neighbourhood of $\partial\Omega$ in order to compute its gradient.  
Note that the chain of equalities \eqref{reassembl_bound_cond} also shows (a posteriori) that the value of $u\otimes u:\nabla n|_{\partial \Omega}$ does not depend on the extension of the normal in a neighbourhood of $\partial \Omega$. 
Clearly, the pressure in \eqref{eq.pressureEq} is always determined up to a constant and its uniqueness can be restored by imposing $\int_{\Omega}p=0$.

\jump

To run the computations in \eqref{reassembl_bound_cond}, we have used $u\in C^1(\overline \Omega)$.  In order to deal with $u\in C^\gamma(\Omega)$, for $\gamma<1$,  we need to interpret \eqref{eq.pressureEq} in the weak sense, that is, we consider a scalar function $p\in C^0(\overline \Omega)$ such that 
\begin{equation}\label{p_weaksol}
-	\int_{\Omega} p\, \Delta \varphi +\int_{\partial \Omega} p \,\partial_n \varphi = \int_{\Omega} u\otimes u : H \varphi,
\qquad \forall  \varphi \in C^2(\overline \Omega),
\end{equation}
where  $H \varphi$ denotes the Hessian matrix of the scalar function $\varphi$.  As usual, \eqref{p_weaksol} is obtained by multiplying the first equation in \eqref{eq.pressureEq} by a test function $\varphi$ and then integrating by parts. Such weak formulation makes sense for every $u,p\in C^0(\overline \Omega)$ and it is actually the equation satisfied by any uniformly continuous couple $(u,p)$  weakly  solving \eqref{E}. 
The rigorous derivation of \eqref{p_weaksol} has been already given in \cite{DLS22}, if test functions whose support is not necessarily compactly contained in $\Omega$ are allowed.  However, since usually weak solutions of~\eqref{E} are defined with compactly supported test functions, we believe that it is worth to dedicate \cref{S:weaksol} to the proof of the equivalence of the two formulations, under quite mild assumptions on the pressure.  Conversely, careful manipulations on weak solutions show that any function $p$ solving \eqref{p_weaksol} gives indeed a pressure weakly solving \eqref{E},  if $u$ is suitably defined  (for instance by \eqref{E_project}). This provides the equivalence of the two approaches: either one starts from a distributional solution $(u,p)$ of \eqref{E} and  deduces the validity of \eqref{p_weaksol}, or one first solves \eqref{p_weaksol} and  obtains a function $p$ which guarantees $\partial_t u + u\cdot \nabla u=-\nabla p$ as distributions.  Further details  are given in \cref{S:weaksol}.

\subsection{Previous results} In the whole discussion below, we will always assume that the incompressible vector field $u$ is $\gamma$-H\"older regular and, when considering a bounded domain $\Omega$, it is also tangent to the boundary.

\jump

Looking at the equation $-\Delta p=\diver \diver (u\otimes u)$ and forgetting the technicalities that might arise from the presence of the boundary,  by standard Schauder's estimates we  get that $p$ is exactly as $\gamma$-H\"older regular as $u$, for every $\gamma\in (0,1)$. However, it has been recently noted that the quadratic structure of the right-hand side $\diver \diver (u\otimes u)$, together with the divergence-free condition $\diver u=0$, allows to increase the H\"older regularity of the pressure up to $2\gamma$.

\jump

At the best of our knowledge, such double regularity has been observed for the first time by L.~Silvestre \cite{SILV} when $\Omega=\R^d$, and then in \cites{CD18,Is2013} by different proofs which in particular generalize the double regularity to the periodic setting $\Omega=\T^d$. More precisely, such results show that, when $\Omega$ is either the whole space $\R^d$ or the torus $\T^d$, the pressure enjoys 
\begin{equation}\label{p_double_regular_litterature}
p\in \left\{\begin{array}{ll}
C^{2\gamma} &\text{ if } 0<\gamma<\frac{1}{2}\\[2mm]
C^{1,2\gamma-1} & \text{ if } \frac{1}{2}<\gamma<1,
\end{array}\right.
\end{equation}
as soon as $u\in C^\gamma$. Soon after, P. Constantin  \cite{C2014} proved that $p\in \loglip(\R^d)$ in the borderline case $\gamma=\frac{1}{2}$ by relying on some useful new local formulas for the pressure on $\R^d$. Let us emphasize that the case $\gamma=\frac12$ is naturally borderline for the regularity \eqref{p_double_regular_litterature}, because of the well-known failure of Schauder's estimates in H\"older spaces with integer exponents. We refer the reader to~\cite{CDF20}, where the double pressure regularity on $\T^d$ has been generalized to any Sobolev or Besov space. 

\jump

Remarkably, the double H\"older regularity of  the pressure on $\R^d$ or $\T^d$, together with several other fine regularity estimates along the flow of $u$,  has been used by P.~Isett in \cite{Is2013} to prove the smoothness of trajectories of  Euler flows when the velocity is strictly less than Lipschitz regular in the spatial variable, thus in a regime in which trajectories are not necessarily unique.  
Very recently, the double regularity of the pressure has been also crucially used in~\cite{DI22} to establish rigorous intermittency-type results in the framework of fully developed turbulence. 
An extensive description of the relevance of intermittency phenomena in the mathematical theory of turbulent flows can be found in the monograph~\cite{F95}. Moreover, on a bounded domain $\Omega\subset \R^d$, where the presence of the boundary adds highly non-trivial complications, the H\"older regularity of the pressure plays a fundamental role in the study of anomalous dissipation, see~\cites{BT18,BTW19,RRS18,RRS2018} and the references therein. 
For these reasons, the study of  the regularity of the pressure up to the boundary is of crucial importance, from both the mathematical and the physical point of view, and our result might be seen as a first step towards the extension of~\cites{DI22,Is2013} to bounded domains.

\jump

As usual when considering boundaries, the above analysis becomes more delicate when $\Omega\subset \R^d$ is a bounded and simply connected domain. Indeed, even the more standard \emph{single} regularity $p\in C^\gamma(\Omega)$  is not a straightforward consequence of Schauder's estimates up to the boundary.
Such regularity was first established by C.~Bardos and E.~Titi in \cite{BT21} for $2$-dimensional domains of class $C^2$ by relying on the global geodesic coordinates. 
Let us also mention that Bardos--Titi's result has been recently (after the first draft of the current paper was online) extended to $3$-dimensional $C^3$ domains in~\cite{BBT23}. The meticulous reader might notice that the boundary condition used in \cites{BT21, BBT23} has an additional term which allows to define the normal derivative of $p$ as an element of $H^{-2}(\partial \Omega)$.  This issue at the boundary is indeed quite subtle, and we postpone the discussion to the next paragraph.  

In our previous work \cite{DLS22}, we  extended Bardos--Titi's result to any dimension. 
Actually, we have been able to also partially double the pressure regularity by proving that, in any dimension $d\geq 2$, if $\Omega\subset \R^d$ is a simply connected open set of class $C^{2,\delta}$ for some $\delta>0$, then the pressure enjoys
\begin{equation}\label{p_double_regular_DLS22}
p\in \left\{\begin{array}{ll}
C^{\gamma}(\Omega) &\text{ if } 0<\gamma<\frac{1}{2}\text{ and } \delta>0\\[2mm]
C^{1,2\gamma-1}(\Omega) & \text{ if } \frac{1}{2}<\gamma<1 \text{ and } \delta=2\gamma-1.
\end{array}\right.
\end{equation}
In the case $\gamma<\frac12$, we exploited the explicit representation formula for $p$ via the Green--Neumann kernel on $\Omega$, while, for $\gamma\in(\frac{1}{2},1)$,  we relied on the already known double regularity  \eqref{p_double_regular_litterature} on~$\R^d$  by suitably extending the vector field $u$ to the whole space. Note that, when $\gamma>\frac{1}{2}$, the requirement $\delta=2\gamma-1$ in \eqref{p_double_regular_DLS22} is necessary, since otherwise the boundary condition $\partial_n p=u\otimes u:\nabla n\in C^{\min\{\delta,\gamma\}}(\partial \Omega)$ would be incompatible with $\nabla p\in C^{2\gamma-1}(\Omega)$. Moreover, once the single $\gamma$-H\"older regularity has been established, by the abstract interpolation argument developed in \cite{CDF20},  an almost double regularity $p\in C^{2\gamma-\eps}(\Omega)$, for any $\eps>0$  arbitrarily small, follows directly as soon as $\partial \Omega\in C^{3,\delta}$ for some $\delta>0$, see \cite{DLS22}*{Theorem 1.3}. 
Let us now explain the Neumann boundary datum in \eqref{eq.pressureEq}.  The derivation of $\partial_n p=u\otimes u:\nabla n$ we have given in  \eqref{reassembl_bound_cond} heavily relies on the assumption $u\in C^1(\overline \Omega)$, which is way above the H\"older regularity we want to consider in this paper. However,  the weak formulation \eqref{p_weaksol} makes sense whenever $u,p\in C^0(\overline \Omega)$. Thus, to circumvent any issue concerning the well-posedness of $\partial_n p$ as a distribution on $\partial \Omega$, what we will do is to show (double) H\"older regularity of the unique zero average solution $p$ to~\eqref{p_weaksol},  where $u\in C^\gamma(\Omega)$ is a given  divergence-free vector field tangent to the boundary.  Slightly abusing terminology, we will say that a weak solution $p$ to \eqref{eq.pressureEq} is a scalar function $p$ satisfying \eqref{p_weaksol}, keeping in mind that going back from \eqref{p_weaksol} to \eqref{eq.pressureEq} is not possible for purely continuous data. As we will discuss in \cref{S:weaksol}, such weak pressure formulation is indeed a natural variational condition satisfied by any couple $(u,p)$ weakly solving~\eqref{E}.  Further comparisons on the difference between the approaches in \cites{BT21,BBT23} versus the one used in the current paper (as well as in \cite{DLS22}) will be discussed in \cref{S:versus_BT}.  

\jump

In this paper, we complete the picture initiated in~\cite{DLS22} by establishing the double H\"older regularity (up to the boundary) of the pressure in the whole range $\gamma\in (0,\frac{1}{2}]$ in any bounded, simply connected and sufficiently regular domain $\Omega\subset\R^d$, including the borderline case $\gamma=\frac12$, in which we prove that the pressure belongs to the H\"older--Zygmund space. Together with \cite{DLS22}, this gives the full double regularity for every $\gamma\in (0,1)$,  in a sufficiently regular simply connected bounded domain $\Omega\subset \R^d$. Our main result is rigorously stated below.

\subsection{New main result}

We let $C^1_*(\Omega)$ be the usual H\"older--Zygmund space, see \cref{s:function_spaces}, and in particular \eqref{d:zygmund_omega}, for  its precise definition. We prove the following

\begin{theorem}\label{thm.main} 
Let $\gamma \in (0,\frac{1}{2}]$  and let $\Omega \subset \mathbb{R}^d$ be a bounded and simply connected domain of class $C^{2, 1}$.  
If $u \in C^{\gamma}(\Omega)$ is a  divergence-free vector field such that $u\cdot n|_{\partial\Omega}=0$, then there exists a unique zero-average solution $p\in C^0(\overline \Omega)$ of~\eqref{p_weaksol} such that
\begin{equation}\label{p_contin_est}
    \|p\|_{C^{2\gamma}(\Omega)} \leq C \|u\|_{C^{\gamma}(\Omega)}^2 \quad\text{for}\ \gamma < \frac{1}{2}
\end{equation}

and 
\begin{equation}\label{p_contin_est_limit}
    \|p\|_{C_*^{1}(\Omega)} \leq C \|u\|_{C^{\frac{1}{2}}(\Omega)}^2
    \quad\text{for}\ \gamma = \frac{1}{2},
\end{equation}
where $C>0$ is a constant depending on $\Omega$ and $\gamma$.
\end{theorem}

The above theorem states that, in the class of $C^0(\overline \Omega)$ functions, there exists a unique zero-average pressure solving \eqref{p_weaksol}, which moreover enjoys \eqref{p_contin_est} and \eqref{p_contin_est_limit}.  However, we emphasize that looking in the larger class $p\in L^1(\Omega)$ does not modify the result, but in this setting one has to be careful about how to define the boundary integral in the left-hand side of \eqref{p_weaksol}. In \cref{R:only_integrable} we show that $p\in L^1(\Omega)$ indeed suffices to have a well-defined trace operator for $p$, a property which we believe being a quite nice outcome of precise cancellations in the quadratic term $\diver \diver (u\otimes u)$.

Similarly to \cite{Is2013},
the proof of \cref{thm.main} is based on the Littlewood--Paley analysis in the frequency space, that we introduce in  \cref{s:function_spaces}.  Moreover, to achieve the above result, we prove the interior and the local boundary regularity estimates separately, see \cref{t:main_interior} and \cref{t:local_boundary}, respectively. 

The interior regularity comes as a consequence of the more general \cref{t:double_Rd_whole_range},  providing the double pressure regularity on the whole $\R^d$  for $u\in C^\gamma_c(\R^d)$.  In the case $\gamma<\frac{1}{2}$, this was indeed already known (as discussed above in \eqref{p_double_regular_litterature}). The main novelty here is for the borderline value $\gamma=\frac{1}{2}$, for which we achieve $p\in C^1_*(\R^d)$, thus improving the $\loglip$ regularity proved in~\cite{C2014} (see \cref{r:loglip_vs_zygmund}).
The main reason for such a sharper regularity in the borderline case lies in the flexibility of the Littlewood--Paley analysis when dealing with  estimates in the borderline H\"older--Zygmund space.

To deal with the boundary  regularity, we pass to the \emph{normal geodesic coordinate} system (see \cref{prop.geometry}) in a neighbourhood of $\partial \Omega$. 
This new coordinate frame, more precisely
the new local induced metric, allows to suitably extend the datum and the unknown to the whole space.  Once on $\R^d$, we 
analyze the corresponding transformed equation by means of the pseudodifferential formalism. The quantitative regularity estimates in $C^{2\gamma}_*$ are then obtained via Littlewood--Paley analysis.
The present strategy is more robust than the one adopted in our previous work \cite{DLS22}, and it seems flexible enough to be helpful also in less regular settings, such as Sobolev or Besov. We refer to \cref{rem.extensionsThm} for a discussion on possible extensions of \cref{thm.main}.

\jump

The  assumption $\partial \Omega \in C^{2,1}$ is needed to work with the boundary normal coordinates. In particular, it ensures that the new local induced metric is Lipschitz, which plays a crucial role in the boundary regularity performed in \cref{s:main_proof_new}. Thus, for rougher domains (say $C^2$),  a different approach at the boundary might be necessary.  We will try to keep track of the minimal regularity of $\partial \Omega$ which is required in every step of the proof, making it easier to localize where the criticality happens.

\jump

The corresponding result for solutions $p$ which are not necessarily average-free directly follows.

\begin{corollary}\label{cor.main} 
Let $\gamma \in (0,\frac{1}{2}]$ and let $\Omega \subset \mathbb{R}^d$ be a bounded and simply connected domain of class $C^{2, 1}$.
If $u \in C^{\gamma}(\Omega)$ is a divergence-free vector field such that $u\cdot n|_{\partial\Omega}=0$, then every weak solution $p\in C^0(\overline \Omega)$ of~\eqref{p_weaksol}  is unique up to constants and satisfies
\begin{equation*}
    \left\|p-\int_{\Omega} p(x)\,dx \right\|_{C^{2\gamma}(\Omega)} \leq C \|u\|_{C^{\gamma}(\Omega)}^2 \quad\text{for}\ \gamma < \frac{1}{2}
\end{equation*}
and 
\begin{equation*}
    \left\|p-\int_{\Omega}p(x)\,dx\right\|_{C_*^{1}(\Omega)} \leq C \|u\|_{C^{\frac{1}{2}}(\Omega)}^2
    \quad \text{for}\ \gamma=\frac12,
\end{equation*}
where $C>0$ is a constant depending on the $\Omega$ and $\gamma$.
\end{corollary}

\subsection{Organization of the paper}In \cref{s:all_main_tools} we discuss all the main technical ingredients that are needed to run our proof: we start by discussing the relation between weak solutions of \eqref{E} and the pressure equation \eqref{p_weaksol}, we introduce the local normal geodesic coordinates,  the Littlewood--Paley analysis, the definition of the H\"older--Zygmund spaces and the pseudodiffrential formalism. In \cref{s:interior} we  prove the interior boundary regularity. \cref{s:boundary} and \cref{s:main_proof_new} are devoted to the boundary analysis: in \cref{s:boundary} we transform the equation, locally at the boundary, in the new coordinate system, while in \cref{s:main_proof_new}  we conclude the proof providing the quantitative stability estimate for the equation in the new coordinate system. We conclude with \cref{rem.extensionsThm} in which we discuss further extensions.

\section{Preliminaries}
\label{s:all_main_tools}

In this section we prove or simply recall results that we will need in our proof.

\subsection{The weak pressure equation}\label{S:weaksol} 

In our previous work~\cite{DLS22}, we showed that, if $u,p\in C^0(\overline \Omega\times (0,T))$ weakly solve~\eqref{E} in the following sense
\begin{equation}
\label{weak_form_test_noncompact}
\int_0^T \int_{\Omega} u\cdot \partial_t \psi + u\otimes u:\nabla \psi +p\diver \psi = \int_0^T \int_{\partial \Omega} p \psi\cdot  n\qquad \forall \psi\in C^1_c(\overline \Omega\times (0,T)),
\end{equation}
then the couple $(u(t),p(t))$, for every fixed $t$-time slice, solves \eqref{p_weaksol}. This is done by simply choosing $\psi(x,t)=\eta(t) \nabla \varphi (x)$,  where $\varphi\in C^2(\overline \Omega)$ is arbitrary.  The formulation \eqref{weak_form_test_noncompact} appears already in the literature, but usually weak solutions of~\eqref{E} are defined by restricting the support of the test vector field $\psi$ to be compactly contained in the open set $\Omega$. Let us now clarify that such two slightly different formulations are equivalent.

\jump

Let $u\in C^0(\overline \Omega\times (0,T))$ be  divergence-free and tangent to the boundary.  The following derivation can be rigorously given under the  very mild assumption $p\in L^1(\Omega\times (0,T))$, but, for the purposes of this presentation, let us assume $p\in C^0(\overline \Omega\times (0,T))$ for simplicity and postpone to \cref{R:only_integrable} the finer discussion when $p$ is merely integrable. Suppose that the couple $(u,p)$ solves 
\begin{equation}\label{weak_form_compact}
\int_0^T \int_{\Omega} u\cdot \partial_t \tilde \psi + u\otimes u:\nabla \tilde \psi +p\diver \tilde \psi = 0, \qquad \forall \tilde \psi\in C^1_c( \Omega\times (0,T)),
\end{equation}
that is, the usual weak formulation of \eqref{E} with compactly supported test vector fields. We want to prove that \eqref{weak_form_test_noncompact} holds. Let $\tilde \Omega\Subset \Omega$ any open set with smooth boundary, $\chi\in C^1_c(\tilde \Omega)$ and $\psi \in C_c^1(\overline \Omega\times (0,T))$. Then $\tilde \psi:=\chi \psi$ is an admissible test for \eqref{weak_form_compact},  and  we obtain
$$
\int_0^T \int_{\Omega}\chi \left( u\cdot  \partial_t  \psi + u\otimes u:\nabla  \psi +p\diver  \psi\right) + \int_0^T \int_{\Omega} (u\cdot \psi) ( u\cdot \nabla \chi)+ p \psi\cdot \nabla \chi= 0.
$$ 
If we let $\chi$ converge to the indicator function of $\tilde \Omega$, then $\nabla \chi$ converges to $-\tilde n \mathcal{H}^{d-1}\llcorner \partial \tilde \Omega$ in the sense of measures, being $\tilde n$ the outward unit normal vector to $\partial \tilde \Omega$ and $\mathcal{H}^{d-1}\llcorner \partial \tilde \Omega$ the standard surface measure.  Thus, we achieve
\begin{equation}\label{eq_weak_interior_domain}
\int_0^T \int_{\tilde \Omega} u\cdot  \partial_t  \psi +  u\otimes u:\nabla  \psi +p\diver  \psi -  \int_0^T \int_{\partial \tilde \Omega}(u\cdot \psi) ( u\cdot \tilde n)= \int_0^T \int_{\partial \tilde \Omega} p \psi\cdot \tilde n, 
\end{equation}
from which, by letting $\tilde \Omega$ invading the whole domain $\Omega$, and recalling that $u$ is tangent to $\partial \Omega$, we conclude the validity of \eqref{weak_form_test_noncompact}. It is clear that, in the above derivation, no assumption on $\partial_n p$ is needed, but only the trace of $p$ on $\partial \Omega$ has to be well defined. In this way,  i.e.\  working only with the weak formulation \eqref{p_weaksol}, we avoid any technical difficulty coming from having a well-defined  boundary condition in \eqref{eq.pressureEq}, which is a very delicate issue,  as noticed in \cites{BT21,BBT23}.

\begin{remark}\label{R:only_integrable}
Assume only $p\in L^1(\Omega\times (0,T))$. Since in the interior we have $-\Delta p=\diver \diver (u\otimes u)$, then by standard Schauder's estimates we have $p(\cdot, t)\in C^\gamma_{\rm loc}(\Omega)$. In particular $p\big|_{\partial \tilde \Omega}$ is  a well defined object for every $\tilde \Omega\Subset \Omega$,  and the formulation \eqref{eq_weak_interior_domain} can be derived exactly as before. 
Differently from above where we assumed $p\in C^0(\overline \Omega\times (0,T))$,  to let $\tilde \Omega$ invading the whole domain $\Omega$ we have to make sure that there is a well-defined notion of $p\big|_{\partial \tilde \Omega}$, when testing with sufficiently regular functions.  This is a quite fine a posteriori property of \eqref{eq_weak_interior_domain}. Indeed the left-hand side has a unique limit as $\tilde \Omega$ goes to $\Omega$, which in turn implies the existence of a well-defined notion of $p\big|_{\partial  \Omega}$ as a linear and continuous operator acting on $\psi \in C^1_c(\overline \Omega\times (0,T))$.  Thus, we have obtained \eqref{weak_form_test_noncompact} with the right-hand side replaced by $\langle p\big|_{\partial \Omega}, \psi\cdot n \rangle$.
Choosing $\psi(x,t)=\eta(t) \nabla \varphi (x)$ as already discussed above, we achieve that, for every $t$-time slice, the couple $(u(t),p(t))\in C^\gamma(\Omega) \times L^1(\Omega)$ solves 
\begin{equation}\label{p_very_very_weak}
-	\int_{\Omega} p\, \Delta \varphi +\langle p \big|_{\partial \Omega},\partial_n \varphi\rangle = \int_{\Omega} u\otimes u : H \varphi,
\qquad \forall \varphi \in C^2(\overline \Omega),
\end{equation}
where $p\big|_{\partial \Omega}$ is the linear operator we have defined above.
In particular, this shows that \cref{thm.main} and \cref{cor.main} could have been stated in the larger class of integrable pressures, not necessarily continuous on $\overline \Omega$, modulo interpreting weak solutions in the sense of \eqref{p_very_very_weak}.
\end{remark}
To complete the picture on the relation between weak solutions of \eqref{E} and the the pressure equation \eqref{p_weaksol}, let us now address the opposite question: does any solution to \eqref{p_weaksol} give an admissible pressure which appears as a gradient in the weak formulation of \eqref{E}? The answer is affirmative, modulo defining weak solutions of \eqref{E} in the right way, that is, when all test vector fields orthogonal (in an $L^2$ sense) to gradients are allowed. By the Helmholtz decomposition on a bounded domain, such tests are given by divergence-free vector fields, tangent to the boundary.  Thus, we say that $u$ is a weak solution of \eqref{E} if 
\begin{equation}\label{E_project}
\int_0^T \int_{\Omega} u\cdot  \partial_t  \psi + u\otimes u:\nabla  \psi  = 0 \qquad \forall \psi\in C^1_c(\overline \Omega\times (0,T)), \, \diver \psi=0, \, \psi\cdot n\big|_{\partial \Omega}=0,
\end{equation}
and moreover, for a.e. $t\in (0,T)$, $u$ is divergence-free and tangent to the boundary, which in the weak sense can be written as $\int_\Omega u(x,t)\cdot \nabla q(x)\,dx=0$ for all $q\in C^1(\overline \Omega)$, and a.e. $t\in (0,T)$.  The formulation \eqref{E_project} is indeed the one which is formally derived from \eqref{E} by multiplying by $\psi$ and integrating by parts, which in particular motivates the choice of the test functions in order to make the pressure disappear. To fix the ideas,  and coherently with the regularity assumptions used in the current paper, let us assume that $u\in L^\infty((0,T);C^\gamma(\Omega))$ solves \eqref{E_project}, with $\gamma<\frac12$.  For a.e. $t\in (0,T)$ we let $p$ be the, unique up to constants,  solution of \eqref{p_weaksol}. Our \cref{thm.main} proves that $p\in L^\infty((0,T);C^{2\gamma}(\Omega))$. We want to prove that \eqref{weak_form_test_noncompact} holds true with this choice of~$p$.
\begin{remark}\label{R:density_test}
By classical density arguments,  the formulation \eqref{E_project} is equivalent to the one in which  $\psi(x,t)=\eta(t) f(x)$, where $\eta\in C^1_c((0,T))$ and $f\in H^1(\Omega)$,  such that $\diver f=0$ and $f\cdot n\big|_{\partial \Omega}=0$.  For the same reason, and since $p$ is continuous up to the boundary in the space variable, the formulation \eqref{weak_form_test_noncompact} does not modify if we restrict to $\psi(x,t)=\eta(t) f(x)$,  where $\eta\in C^1_c((0,T))$ and $f\in H^1(\Omega)$.  Also in \eqref{p_weaksol} we can allow every $\varphi\in H^2(\Omega)$.
\end{remark}
Thus, let us fix any test vector field of the form $\psi(x,t)=\eta(t) f(x)$,  where $\eta\in C^1_c((0,T))$ and $f\in H^1(\Omega)$. Solve 
\[
\left\{
    \begin{array}{rcll}
        \Delta \varphi &=& \diver f & \text{ in }\Omega\\[1ex]
        \partial_n \varphi &=&f\cdot n &\text{ on } \partial \Omega,
    \end{array}
    \right.
\] 
which admits a solution $\varphi \in H^2(\Omega)$,  unique up to constants, and decompose $f=f-\nabla \varphi + \nabla \varphi=: f_{\rm div}+ \nabla \varphi$. Clearly $f_{\rm div}\in H^1(\Omega)$,  $\diver f_{\rm div}=0$ and $f_{\rm div}\cdot n \big|_{\partial \Omega}=0$. By using \eqref{E_project}, together with $\int_{\Omega} u\cdot \nabla \varphi=0$ for a.e. $t\in (0,T)$, we have 
\begin{align*}
&\int_0^T \int_{\Omega} \left( u\cdot \partial_t \psi + u\otimes u:\nabla \psi +p\diver \psi \right)- \int_0^T \int_{\partial \Omega} p \psi\cdot  n\\
&=\int_0^T \int_{\Omega}\left( u\cdot f_{\rm div} \eta'+ u\otimes u:\nabla f_{\rm div} \eta\right) \\
&+ \int_0^T  \eta \left( \int_{\Omega} u\otimes u:\nabla (\nabla  \varphi)  + p \diver (\nabla \varphi) - \int_{\partial \Omega}  p \nabla \varphi \cdot  n \right)\\
&= \int_0^T\eta \left(  \int_{\Omega} u\otimes u:H \varphi  + p \Delta \varphi- \int_{\partial \Omega} p \partial_n \varphi\right) =0,
\end{align*}
where the last term in the above chain of equalities vanishes since we have chosen $p$ solving \eqref{p_weaksol}. This, together with the equivalence observed in \cref{R:density_test}, proves that the couple $(u,p)$ solves  \eqref{weak_form_test_noncompact}.

Summing up,  we have rigorously proved that starting from \eqref{weak_form_compact}, with the a priori assumption $p\in L^1(\Omega\times (0,T))$, then $p$ necessarily solves \eqref{p_very_very_weak}, that is the generalized version of \eqref{p_weaksol} when $p$ is not necessarily a priori continuous up to the boundary.  Conversely, we have also proved that, if $u$ is a weak solution to \eqref{E} in the sense of \eqref{E_project}, then the solution $p$ to \eqref{p_weaksol} (or, more generally, of \eqref{p_very_very_weak}) gives a pressure which, when coupled with $u$, solves \eqref{E} in the sense of  \eqref{weak_form_test_noncompact}, and thus also in the one of \eqref{weak_form_compact}.

\subsection{Normal geodesic coordinates}\label{sec.geometry}

Let  $\Omega \subset \mathbb{R}^d$ be a domain of class $C^{2,\delta}$ for some $\delta \in(0,1]$, endowed with the constant Euclidean metric $g^{ij}=g_{ij}=\delta_{ij}$ for all $i, j \in \{1, \dots, d\}$, where  $g_{ij}=\langle e_i, e_j \rangle$ is the scalar product between the basis vectors $e_i$ and $e_j$.

Let  $x_0 \in \partial\Omega$ and fix a small neighborhood $U$ of $x_0$ in $\R^d$. Let $U_0 \coloneqq U \cap \partial\Omega$, which is a neighborhood of $x_0$ in $\partial\Omega$, and consider a local coordinate chart for $\partial\Omega$ at $x_0$, \textit{i.e.}, a couple $(U_0,\varphi)$, where $\varphi \colon U_0 \to V_0=:\varphi(U_0) \subset \mathbb{R}^{d-1}$ is a $C^{2,\delta}$-diffeomorphism. We write $\varphi(x)=(\theta^1(x),\dots, \theta^{d-1}(x))$ and call  the $\theta^i$'s the \emph{local coordinates} on $\partial\Omega$ around $x_0$. Up to reducing the neighborhood $U$, one can assume that $U$ is a tubular neighborhood of $\partial\Omega$, that is, $U=\{x + tn(x) : -c \leq t \leq c,\ x \in \Sigma \subset \partial\Omega\}$ for some $c>0$, some portion of the boundary $\Sigma \subset \partial\Omega$ and where $n(x)$ stands for the inward pointing unit normal to $\partial\Omega$ at $x$.   

We can express the Euclidean metric $g$ in any new coordinate system $y^i$ (and where we  let $\frac{\partial}{\partial y^i}$ be the associated dual vector field) using the following change of variable formula for all $x \in \mathbb{R}^d$
\[
    g_{ij}(x)=\left\langle  \frac{\partial x}{\partial y^i}, \frac{\partial x}{\partial y^j}\right\rangle = \sum_{k=1}^d\frac{\partial x^k}{\partial y^i}\frac{\partial x^k}{\partial y^j}. 
\]

In our proof we will use the so-called \emph{normal geodesic coordinates} which are defined as follows. For any point $x\in U$, let us denote $\pi(x)$ the orthogonal projection (for the Euclidean metric) of $x$ on $\partial\Omega$ and write $r(x)=\operatorname{dist}(x,\partial\Omega)$. Shifting the indices, we also let $(\theta^i(x))_{2 \leq i \leq d}$ be  the coordinates of $\pi (x)$ on $\partial\Omega$. Our new coordinates are  $y^1=r$  and $y^j=\theta^j(x)$, $j\geq 2$, which define local coordinates in~$U$. 

\begin{prop}\label{prop.geometry}  Let $\Omega\subset \R^d$ a bounded domain of class $C^{2,\delta}$,  $\delta\in (0,1]$. With the above notation, the following statements hold true.
\begin{enumerate}
\item There exists a symmetric matrix $g^{\theta\theta}=[g^{\theta^i\theta^j}]_{2\leq i,j \leq d}$ with $C^{0,\delta}$ coefficients such that the inverse metric $g$ in the new coordinates $(y^i)_{1\leq i \leq d}$ reads as 
\[
    g=\begin{bmatrix} 1 & 0 \\
    0 & g^{\theta\theta}\end{bmatrix}.
\]
\item There exists $c>0$ such that $\xi^Tg(r,\theta)\xi \geq c |\xi|^2$ for all $\xi \in \mathbb{R}^d$ and all $(r,\theta)\in U$.
\end{enumerate}
\end{prop}

\begin{proof} Let us prove the two claims separately.

\jump

\textit{Proof of (1)}. We are going to prove that $\bar g = [g_{ij}]_{1 \leq i, j \leq d}$ in the coordinates $y^i$ assumes the block-diagonal form, therefore the inverse $g=[g^{ij}]_{1 \leq i, j \leq d}$ will have the same form. Note that by definition of $n$, $\pi$ and $r$, there holds $x=r(x)n(\pi(x)) + \pi(x)$ for any $x\in U$. 

We start by proving that $g_{11}=1$. With the notation $y^1=r$ and $(y^{2}, \dots, y^{d})=\theta$, we have to prove $\left\vert\frac{\partial x(r,\theta)}{\partial r}\right\vert^2=1$ for all $(r,\theta) \in V$. Observe that, for all $h$ small enough (in order to remain in the tubular neighborhood), there holds 
\[
    x(r+h,\theta)-x(r,\theta)=hn(\pi(r,\theta)),
\]
therefore $\frac{\partial x(r,\theta)}{\partial r}=n(\pi(r,\theta))$, which has norm $1$. 

Let us now prove that, for all $i \in \{2, \dots , d\}$, there holds $g_{1i}=0$, which is equivalent to $\left\langle\frac{\partial x(r,\theta)}{\partial \theta^i},n(\pi(r,\theta))\right\rangle = 0$. Let $(e_i)_{2 \leq i \leq d}$ be the Euclidean basis of $\mathbb{R}^{d-1}$. Let $k\geq 2$ and $\eta = (0, \dots, 0, \eta_k, 0, \dots, 0)$. We can therefore write
\begin{align*}
    x(r,\theta + \eta) &= n(r, \theta + \eta) + \pi(r,\theta + \eta) = n(r,\theta + \eta) + \varphi^{-1}((\theta^i + \delta_{ik}\eta_k)e_i),
\end{align*}
so that, using Taylor's formula at order one, we find 
\[
    x(r,\theta + \eta) - x(r,\theta) = \eta_k \partial_{\theta^k}n(r,\theta) + \eta_k \partial_k(\varphi^{-1})(\theta^ie_i)  + \mathcal{O}(\eta_k^2).
\]
Differentiating the equality $|n(r,\theta)|^2=1$ in $\theta^k$, we see that $\partial_{\theta^k}n(r,\theta)$ is orthogonal to $n(r,\theta)$. We observe that $\partial_k(\varphi^{-1})(\theta^ie_i)$ is also orthogonal to $n(r,\theta)$. Recall that $x(r,\theta)=r n(r,\theta)+\pi(r,\theta)$ and note that $g_{\theta\theta}$ is a matrix whose coefficients are functions of $\nabla x(r,\theta)$, therefore are $C^{0,\delta}$ as well as their inverse, \textit{i.e.}, $g^{\theta\theta} = (g_{\theta\theta})^{-1} \in C^{0,\delta}$,  completing the proof of the claim.

\smallskip 

\textit{Proof of (2)}.
Let us observe that, since $g$ is block-diagonal, it is enough to prove that 
\[
    \eta^Tg^{\theta\theta}(r,\theta)\eta \geq c |\eta|^2
    \quad 
    \text{for all}\
    \eta \in \mathbb{R}^{d-1}.
\]
Letting $J$ denote the inverse of the matrix  $(\partial_{\theta^j} x^i)_{2\leq i,j\leq d}$,  we have $g^{\theta\theta}=J^TJ$, so that we only need to explain why $\eta^Tg^{\theta\theta}(r,\theta)\eta = |J\eta |^2 \geq c >0$ when $|\eta|=1$. But the latter property readily follows since $J$ is invertible at every $(r,\theta) \in V$ and the unit sphere is a compact set. 
\end{proof}

We can write the Laplacian in the new coordinates as 
\begin{equation}
    \label{eq.geodesicLaplace}
    -\Delta p = -(\det(g))^{-1/2}\partial_i\left((\det(g))^{1/2}g^{ij}\partial_j p\right)= - \frac{1}{G} \partial_{i}(G g^{ij}\partial_jp),
\end{equation}
  where $G(r,\theta)\coloneqq\sqrt{\det g(r,\theta)}$ and  $\partial_i$ denote the derivatives in the new coordinates $r,\theta$. 
Moreover, in the new local coordinates system, the boundary condition $u \cdot n = 0$ on $\partial\Omega$ reads as $u^r(0,\theta)=0$ for all $\theta$ and, similarly, the boundary condition $\partial_n p = 0$ on $\partial\Omega$ is equivalent to $\partial_rp(0,\theta)=0$.
Finally, for a vector field $F=F^ie^i$ there holds
    \begin{equation}
\label{eq.divForumation}
        \diver F = (\det g)^{-1/2}\partial_i\left(F^i(\det g)^{1/2}\right),
    \end{equation}
 from which we get
    \begin{equation}
        \label{eq.doubleDiv}
        \diver \diver(u\otimes u)=(\det g)^{-1/2}\partial_{ij}^2\left((\det g)^{1/2}u^iu^j\right)=\frac{1}{G}\partial_{ij}^2\left(G u^iu^j\right). 
    \end{equation}

\subsection{Littlewood--Paley analysis, H\texorpdfstring{\"o}{ö}lder and H\texorpdfstring{\"o}{ö}lder--Zygmund spaces}\label{s:function_spaces}

We introduce a smooth Littlewood--Paley partition of the unity 
\[
    1 = \sum_{N \in 2^{\mathbb{N}}} \mathbf{P}_N(\xi)
\]
as follows. Let $\varphi$ be some smooth bump function which is non-negative, radially symmetric, supported on $\{|\xi|\leq 2\}$ and such that $\varphi(\xi)=1$ for all $|\xi|\leq 1$. We let $\mathbf{P}_1(\xi) \coloneqq \varphi (\xi)$ and  $\mathbf{P}_N(\xi)=\varphi(N^{-1}\xi)-\varphi(2N^{-1}\xi)$ for all $N\geq 2$. Therefore $\mathbf{P}_N$ is supported on $|\xi| \sim N$. Let $u \in \mathcal{S}'$ be a tempered distribution. For all dyadic integers $N \in 2^{\mathbb{N}}$ we define
\[
    u_N \coloneqq \mathbf{P}_Nu \coloneqq \mathcal{F}^{-1}(\mathbf{P}_N)* u,
\]
which is just $\hat{u}_N(\xi)=\mathbf{P}_N(\xi)\hat{u}(\xi)$ on the Fourier side. 
Because the Fourier transform of $u_N$ is compactly supported, $u_N$ is smooth. Also, due to the partition of unity property, there holds 
\[
    u = \sum_{N \in 2^{\mathbb{N}}}u_N.
\]
Let us also denote
\[
    u_{\leq N} \coloneqq \sum_{M\leq N} u_M \quad \text{ and } \quad u_{\geq N} \coloneqq \sum_{M\geq N} u_M,
\]
where the summation is on dyadic $M\in 2^{\mathbb{N}}$. In this article, when a summation is taken on capitalized letters, it is always assumed that the summation is on dyadic integers only. For a much more extensive presentation, we refer the reader to \cite{BCD}.

\jump

Using the frequency localization of $u_N$ and Young's inequality, one can infer the following very useful estimates that we shall constantly use in \cref{s:main_proof_new}. 

\begin{theorem}[Bernstein's Theorem~\cite{tao}*{Appendix A}]\label{thm.bernstein} Let $s \in \mathbb{R}$ and $p,q\in [1,\infty]$ such that $p\leq q$. 
Then, the following hold:
\begin{enumerate}[(i)]
    \item $\displaystyle\|u_N\|_{L^q} \lesssim  N^{d \left(\frac{1}{p}-\frac{1}{q}\right)}\|u_N\|_{L^p}$;
    \item $\| |\nabla| ^s u_N\|_{L^p} \sim N^s\|u_N\|_{L^p}$ if $N\geq 2$ and $\||\nabla| ^su_1\|_{L^p} \lesssim \|u_1\|_{L^{\infty}}$, where $|\nabla|^s$ is the Fourier multiplier  $|\xi|^s$;
    \item $\|u_{\leq N}\|_{L^p} \lesssim \|u\|_{L^p}$ and $\|u_{\geq N}\|_{L^p} \lesssim \|u\|_{L^p}$. 
\end{enumerate} 
In the latter inequalities, the implicit constants depend on $d$, $p$, $q$ and $s$ but not on~$N$. 
\end{theorem}

The definition of the smooth truncation $\mathbf{P}_N$ implies that  $\mathbf{P}_N\mathbf{P}_M = 0$, unless $\frac{1}{4}M \leq N \leq4M$, which we abbreviate as $N \sim M$. We also write $N \gg M$ when $N \geq 8M$.

\jump

If $N \geq M$, then $u_Nv_M$ has frequency in $\{N- M \leq |\xi| \leq N+M\}$. In the case $N \sim M$, this means that $u_Nv_M$ is frequency supported in $\{|\xi| \lesssim N\}$, so that
$\mathbf{P}_K(u_Nv_M) = 0$ unless $K \lesssim N$. The other interesting case is when $N \gg M$, in which case $u_Nv_M$ is frequency localized in $\{|\xi|\sim N\}$, so that $\mathbf{P}_K(u_Nv_M)=0$ unless $K\sim N$. These considerations will be frequently used  in \cref{s:main_proof_new}. We shall also repeatedly use  the following simple facts about dyadic sums
\[
    \sum_{1\leq M<N}1 \lesssim \log(N), \quad \sum_{M<N}M^s \lesssim N^s \quad\text{and}\ \sum_{M>N}M^{-s} \lesssim N^{-s},
\]
for any $s>0$, where the implicit constants in the last two estimates might also depend on $s$.

\jump

We now recall some basic facts about H\"older--Zygmund spaces. For a more detailed account, we refer to \cite{taylor}. For $s \in \mathbb{R}$, we define 
\begin{equation}
    \label{eq.normHolderZ}
    \|u\|_{C^s_*} \coloneqq \sup_{N\in 2^{\mathbb{N}}} N^{s}\|u_N\|_{L^{\infty}},
\end{equation}
which is nothing but the Besov norm $\|u\|_{B^s_{\infty,\infty}}$.
It is well known that $C^s(\mathbb{R}^d)=C^s_*(\mathbb{R}^d)$ whenever $s\in \R^+\setminus \N$, where $C^s(\mathbb{R}^d)$ denotes the usual space of $s$-H\"older continuous functions, while $C^s(\mathbb{R}^d)\subset C^s_*(\mathbb{R}^d)$ with strict inclusion if $s\in \N$, see  \cite{taylor}*{Appendix~A} for instance.

\jump

Let us also recall the following useful facts about the borderline case $s=1$.

\begin{prop}
The space $C^1_*(\R^d)$ is the classical Zygmund space, and the norm~\eqref{eq.normHolderZ} is equivalent to 
    \begin{equation}\label{d:zygmund_difference_quotioent}
        \|u\|_{L^{\infty}(\R^d)} + \sup_{x \in \mathbb{R}^d}\sup_{h\in \R^d,\,h\neq 0} \frac{|u(x+h)+u(x-h)-2u(x)|}{|h|}.
        \end{equation}
    Moreover, $C^1_*(\R^d) \subset\loglip(\R^d)$ with continuous embedding, where $\loglip(\R^d)$ denotes the space of functions such that 
    \begin{equation}\label{d:loglip}
       \|u\|_{\loglip(\R^d)}\coloneqq  \|u\|_{L^{\infty}(\R^d)} + \sup_{\substack{x,y\in \R^d \\x \neq y}} \frac{|u(x)-u(y)|}{|x-y|\left(1+\left| \log |x-y|\right| \right)}<\infty.
    \end{equation}
\end{prop}

\begin{proof} 
For \eqref{d:zygmund_difference_quotioent}, we refer to \cite{taylor}*{Appendix A}. To prove the embedding, define  $w(h)\coloneqq(u(x+h)-u(x))/|h|$ for $h\neq 0$. Since $u\in C^1_*(\R^d)$, we have 
\[
\left| w(h)-w\left(\frac{h}{2}\right)\right|=\frac{\left| u(x+h)+u(x)-2u\left(x+\frac{h}{2}\right)\right|}{|h|}\leq C\|u\|_{C^1_*(\R^d)},
\]
from which we deduce 
\[
|w(2^{k+1}h)-w(2^k h)|\leq C\|u\|_{C^1_*(\R^d)},\quad \forall k\geq 0.
\]
The constant $C>0$ may vary from line to line in the next computations, but it is important that it will always be independent of $h$ and $k$.

\jump

Choose $k_0\in \N$ such that $e^{k_0}|h|\sim 1$ and add the previous inequality for all $0\leq k<k_0$, getting
\[
|w(2^{k_0}h)-w(h)|\leq k_0C\|u\|_{C^1_*(\R^d)}\leq C\|u\|_{C^1_*(\R^d)} \left| \log |h| \right|.
\]
Thus we finally achieve
\begin{align*}
|w(h)|&\leq |w(2^{k_0}h)|+|w(2^{k_0}h)-w(h)|\leq 2^{-k_0}|h|^{-1}\|u\|_{L^\infty(\R^d)}+ C\|u\|_{C^1_*(\R^d)} \left| \log |h| \right|\\
&\leq C\left(\|u\|_{L^\infty(\R^d)}+ \|u\|_{C^1_*(\R^d)}\right) \left(1+\left| \log |h| \right| \right),
\end{align*}
which, in terms of $u$, reads as
\[
|u(x+h)-u(x)|\leq C\left(\|u\|_{L^\infty(\R^d)}+ \|u\|_{C^1_*(\R^d)}\right)|h| \left(1+\left| \log |h| \right| \right)
\]
yielding the claimed embedding.
\end{proof}
We can use the equivalent norm \eqref{d:zygmund_difference_quotioent} to naturally define the Zygmund space $C^1_*$ on any bounded domain $\Omega$ by setting
\begin{equation}\label{d:zygmund_omega}
   \|u\|_{C^1_*(\Omega)}\coloneqq \|u\|_{L^{\infty}(\Omega)} + \sup_{\substack{
   x,\, x+h,\, x-h\,\in\, \Omega \\ h\neq 0}} \frac{|u(x+h)+u(x-h)-2u(x)|}{|h|}.
\end{equation}
Also the $\loglip$ norm \eqref{d:loglip}  extends to any bounded domain by setting
   \begin{equation}\label{d:loglip_omega}
       \|u\|_{\loglip(\Omega)}\coloneqq  \|u\|_{L^{\infty}(\Omega)} + \sup_{\substack{x,y\in \Omega \\x \neq y}} \frac{|u(x)-u(y)|}{|x-y|\left(1+\left| \log |x-y|\right| \right)}<\infty.
    \end{equation}
\begin{remark}\label{r:loglip_vs_zygmund}
Note that the function $u\colon[-1,1] \rightarrow \R$ given by $u(x)=-|x|\log|x|$, satisfies $u\in  \loglip([-1,1] )\setminus C^1_*([-1,1] )$, proving that $C^1_*([-1,1] ) \subsetneq \loglip([-1,1] )$.
\end{remark}

Finally, as it has been first observed by J.~M.~Bony \cite{B81}, we recall that the product of a distribution in $C^{r}_*(\R^d)$ for $r<0$ with a function in $C_*^{s}(\R^d)$ defines a distribution if $s+r>0$, see \cite{GIP15}*{Lemma 2.1} for instance.

\begin{lemma}[Product of distributions]\label{lem.holderProduct} 
Let $s,r\in\R$ be such that $r<0<s$ and $r+s>0$. 
If  $u \in C^r_*$ and $v\in C^s_*$,  then $uv \in C^{r}_*$, with 
\begin{equation}\label{est_prod_holder_distrib}
    \|uv\|_{C^r_{*}} \lesssim \|u\|_{C^r_*}\|v\|_{C^s_*}. 
\end{equation}
\end{lemma}

Notice that defining the product of two distributions is not trivial, and indeed it involves a \emph{Bony decomposition} together with \emph{paradifferential} calculus. However, we remark that we are going to use \cref{lem.holderProduct} on products of bounded and continuous functions, thus in this case $uv$ is trivially defined and the estimate \eqref{est_prod_holder_distrib} becomes an easy exercise in Littlewood--Paley analysis.

\subsection{Pseudodifferential operators and symbols}\label{s:pseudodiff_formalism}
In what follows we introduce the pseudodifferential formalism from \cite{taylorIII}. We will repeatedly use the notation $\langle \xi\rangle:=\sqrt{1+|\xi|^2}$.

\begin{definition}[Classes of symbols] Let $m\in \mathbb{R}$ and $\delta \in [0,1]$. A symbol $a\in C^{\infty}_{x,\xi}$ is in~$S^m_{1,\delta}$ if 
\[
    |\partial^{\alpha}_{\xi}\partial_x^{\beta} a(x,\xi)| \leq C(\alpha, \beta) \langle \xi \rangle ^{m - |\alpha| + \delta |\beta|}.
\]
We let $S^m \coloneqq S^m_{1,0}$ be the space of classical symbols of order $m$. 
\end{definition}
When dealing with limited regularity in the variable $x$, one has the following  generalization. 

\begin{definition}[Class $C^s_*S^m_{1,\delta}$] A symbol $a$ with regularity $C_*^s$ in $x$ is in the class $C^s_*S^m_{1,\delta}$ when 
\[
    \left\|\partial_{\xi}^{\alpha} a(\cdot,\xi)\right\|_{C_*^s} \leq C(s,\alpha) \langle \xi \rangle ^{m-|\alpha| +s\delta}.
\]
\end{definition}

\begin{definition}[Quantization of a symbol $a$] Let $a \in C^r_*S^{m}_{1,\delta}$. The \emph{quantization} of $a$ is the operator $\operatorname{Op}(a) : \mathcal{S} \to \mathcal{S}$ defined for every $u \in \mathcal{S}$ by 
\[
    \operatorname{Op}(a)u(x)=\int_{\mathbb{R}^d \times \mathbb{R}^d} e^{i\xi \cdot (x-y)}a(x,\xi)u(y)\,\mathrm{d}\xi\,\mathrm{d}y = \int_{\mathbb{R}^d}K(x,y)u(y)\,\mathrm{d}y, 
\]
where $K(x,y)=\displaystyle\int_{\mathbb{R}^d}e^{i\xi \cdot (x-y)}a(x,\xi)\,\mathrm{d}\xi$. 
\end{definition}

This quantization is such that, for example, if $a(x,\xi)=ib(x)\xi_k$ then 
\[
    \operatorname{Op}(a)u(x)=b(x)\partial_k u(x).
\] 
Other quantizations exist, but this one (which is the classical one) fits to our problem. 

\jump

It is well-known that operators in the class $\operatorname{Op}(S^m_{1,\delta})$ enjoy good continuity bounds in $W^{s,p}$ or $\mathcal{C}^s_*$ spaces,  as known from the more general Calder\'on--Vaillancourt Theorem.

\begin{theorem}[\cite{taylorIII}*{Chapter~13, Corollary 9.2}]\label{th.calderonVaillancourt} Let $m\in \mathbb{R}$, $\delta\in[0,1)$ and $a \in S^m_{1,\delta}$. Then, for any $s \in \mathbb{R}$, there holds
\[
    \operatorname{Op}(a) : C_*^{s+m}\to C_*^{s}.
\]
\end{theorem}


Operators associated to symbols in the class $S^{m}_{1,\delta}$ enjoy nice composition properties.

\begin{theorem}[\cite{taylorII}*{Chapter 7,  Proposition 3.1}]\label{thm.calculus} Let $A=\operatorname{Op}(a)$ and $B=\operatorname{Op}(b)$ two pseudodifferential operators with symbols $a \in S^{m}_{1,\delta}$ and $b\in S^n_{1,\delta}$, for some $n, m \in \mathbb{R}$ and $\delta \in [0,1)$. Then, the composition $C\coloneqq A\circ B$ is a pseudodifferential operator $C=\operatorname{Op}(c)$, where $c\in S^{m+n}_{1,\delta}$. Moreover, for any integer $N \geq 0$,  it holds
\[
    c(x,\xi) -\sum_{|\alpha| \leq N} \frac{i^{|\alpha|}}{\alpha !}\partial^{\alpha}_{\xi} a(x,\xi) \partial^{\alpha}_{x}b(x,\xi) \in S^{m+n-(1-\delta)(N+1)}_{1,\delta}. 
\]
\end{theorem}

We will need to invert elliptic operators. In our context we will use the following definition of elliptic operators.

\begin{definition}[Elliptic symbol on a compact set] We say that a symbol $a \in S^m_{1,\delta}$ is elliptic on a compact set $K \subset \mathbb{R}^d$ if there exist $R>0$ and $c>0$ such that, for all $|\xi|\geq R$ and all $x \in K$, there holds
\[
    |a(x,\xi)| \geq c \langle \xi \rangle ^{m}.
\]
\end{definition}

With this definition, an elliptic operator can be (locally) inverted modulo smooth functions.

\begin{theorem}[Elliptic operator inversion,  \cite{taylor}*{Section~0.4}]\label{thm.elliptic} Let $m>0$ and $\delta \in [0,1)$ and let  $A \in \operatorname{Op}(S^{m}_{1,\delta})$  be an elliptic operator on a compact set $K$. Then, there exist $B \in \operatorname{Op}(S^{-m}_{1,\delta})$ and $C \in \operatorname{Op}(S_{1,\delta}^{-2(1-\delta)})$ such that 
\[
    B\circ A=\operatorname{Op}(\chi(x)) + C,
\]
for some localization function $\chi$ supported in $K$
\end{theorem}

\begin{proof} We provide a proof for the sake of completeness. Note that this result is a variant of the usual proof of inversion of elliptic operators modulo smooth functions~\cite{taylorII}*{Chapter 7.4}. 

We seek an operator $B=\operatorname{Op}(b)$, where $b=b_{-m}+b_{-m-(1-\delta)}$ for some $b_{-m-(1-\delta)}\in S^{-m-(1-\delta)}_{1,\delta}$ to be found. The first term $b_{-m}$ is actually easy to identify. In order to write it, let us introduce  some cutoff function $\chi$ supported in $K$. Also, let $\psi$ be a cutoff function which takes values $1$ for $|\xi| \geq 2R$ and which is $0$ for $|\xi|\leq R$. Let us set
\[
    b_{-m}(x,\xi)=\frac{\chi(x)\psi(\xi)}{a(x,\xi)},
\]
which, thanks to our assumptions on $a$ and the chain rule, is a symbol in $S^{-m}_{1,\delta}$.
Observe that $b_{-m}(x,\xi)a(x,\xi)=\chi(x)\psi(\xi)$, so \cref{thm.calculus} applied for $N=1$ to $\operatorname{Op}(b_{-m})$ and $\operatorname{Op}(a)$ gives 
\[
    \operatorname{Op}(b_{-m}) \circ \operatorname{Op}(a) = \operatorname{Op}(\chi(x)\psi(\xi)) + \operatorname{Op}(i\nabla_{\xi}b_{-m}\cdot\nabla_x a) + C_0,
\]
where $C_0 \in \operatorname{Op}\left(S_{1,\delta}^{-2(1-\delta)}\right)$. 

Now observe that, again by applying \cref{thm.calculus} with $N=0$ to $\operatorname{Op}(b_{-m-(1-\delta)}) \circ \operatorname{Op}(a)$, there exists $C_1 \in \operatorname{Op}\left(S_{1,\delta}^{-2(1-\delta)}\right)$ such that 
\begin{align*}
    \operatorname{Op}(b) \circ \operatorname{Op}(a) &= \operatorname{Op}(\chi(x)\psi(\xi)) + \operatorname{Op}(i\nabla_{\xi}b_{-m}\cdot\nabla_x a) + C_0  + \operatorname{Op}(b_{-m-(1-\delta)}a) + C_1 \\
    & = \operatorname{Op}(\chi(x)\psi(\xi)) + C_0 + C_1,
\end{align*}
provided that we take $b_{-m-(1-\delta)}\coloneqq - \frac{i}{a}\nabla_{\xi}b_{-m}\cdot \nabla_xa$, which is a symbol in $S_{1,\delta}^{-m-(1-\delta)}$.

Let us finally observe that $\chi(x)\psi(\xi)=\chi(x) + \chi(x)(\psi(\xi)-1)$ and that $c_2 \coloneqq \chi(x)(\psi(\xi)-1)$ belongs to any $S^{-M}_{1,\delta}$ class ($M>0$), by choosing $\psi$ appropriately. We can thus write 
\[
    \operatorname{Op}(b) \circ \operatorname{Op}(a) = \operatorname{Op}(\chi(x)) + C,
\]
where $C=C_0+C_1+\operatorname{Op}(c_2)$ as wanted. 
\end{proof}

\subsection{Reduction to smooth functions} Since it will be very convenient to work with a $C^1(\overline \Omega)$ vector field $u$, and thus to justify all our computations in the classical sense, let us consider $u_{\varepsilon}$, the smooth regularization of $u$ given by \cite{DLS22}*{Lemma 2.1 and Remark 2.3}, that we recall here.

\begin{lemma}[Velocity approximation]\label{l:approx}
Let $d\ge2$ and let $\Omega\subset\R^d$ be a bounded and simply connected domain of class $C^{2,\delta}$, for some $\delta\in (0,1)$. Let $\gamma \in (0,1)$ and let $u\in C^{\gamma}(\Omega)$ be such that $\diver  u=0$ and $u\cdot n\vert_{\partial\Omega}=0$.  Then, there exists a family $(u_{\varepsilon})_{\eps>0} \subset C^{\infty}(\Omega)\cap C^{1,\delta}(\Omega)$ such that
$u_{\varepsilon}\to u\ \text{in}\ C^0(\overline \Omega)$
as $\eps\to 0^+$,
$\diver   u_{\varepsilon}=0$ and $u_{\varepsilon}\cdot n \vert_{\partial \Omega}=0
$ for all $\eps>0$, and
\begin{equation*}
\sup_{\eps>0}\|u_{\varepsilon}\|_{C^{\gamma}(\Omega)} \leq C \|u\|_{C^{\gamma}(\Omega)}
\end{equation*} 
for some constant $C>0$.
\end{lemma}

The proof of \cref{thm.main} relies on the following

\begin{theorem}\label{t:main_approx} Let $\gamma \in (0,\frac{1}{2}]$ and $\Omega \subset \mathbb{R}^d$ be a simply connected bounded domain of class $C^{2,1}$. For $\delta\in (0,1)$ let $u \in C^{1,\delta}( \Omega)\cap C^\infty (\Omega)$ be a divergence-free vector field such that $u\cdot n|_{\partial\Omega}=0$.  Then there exists  a unique zero-average solution $p\in C^{1,\delta}(\Omega)$ to~\eqref{eq.pressureEq} such that
\begin{equation}\label{p_contin_est_smooth}
    \|p\|_{C^{2\gamma}_*(\Omega)} \leq C \left( \|u\|_{C^{\gamma}(\Omega)}^2 +\|p\|_{L^\infty(\Omega)}\right),
\end{equation}
for some constant $C>0$ which depends on $\Omega$ and $\gamma$ only.
\end{theorem}

Thanks to the above result, the proof of \cref{thm.main} follows by standard tools in analysis. We give the details for the reader's convenience.
\begin{proof}[Proof of \cref{thm.main}]
We divide the proof in two steps: we first eliminate the term $\|p\|_{L^\infty(\Omega)}$ from the right-hand side of \eqref{p_contin_est_smooth}, and then we show how such a quantitative continuity estimate leads to the existence of solutions to \eqref{p_weaksol} for any $u\in C^\gamma(\Omega)$ not necessarily smooth.

\smallskip

\textit{Step 1}.
We prove that, if $u \in C^{1,\delta}( \Omega)\cap C^\infty (\Omega)$, then the unique zero-average solution  $p$ of~\eqref{eq.pressureEq} (which always exists thanks to \cref{t:main_approx}) enjoys
\begin{equation}\label{p_contin_est_no_Linfty}
    \|p\|_{C^{2\gamma}_*(\Omega)} \leq C  \|u\|_{C^{\gamma}(\Omega)}^2.
\end{equation}
Indeed, suppose that \eqref{p_contin_est_no_Linfty} does not hold. Then, for all $k\in \N$, we can find a divergence-free vector field $u_k$ and a scalar function $p_k$ solving 
\[
   \left\{
    \begin{array}{rcll}
        -\Delta p_k &=& \diver \diver (u_k\otimes u_k) & \text{ in }\Omega\\[1ex]
        \partial_n p_k&=&u_k\otimes u_k : \nabla n & \text{ on } \partial \Omega,
    \end{array}
    \right.
\]
with $\int_\Omega p_k=0$ and
\[
1=\|p_k\|_{C^{2\gamma}_*(\Omega)}\geq k \|u_k\|^2_{C^{\gamma}(\Omega)}.
\]
In particular, $u_k\rightarrow 0$ in $C^\gamma(\Omega)$, which, by passing to the limit in the weak formulation \eqref{p_weaksol}, implies that the uniform limit of $(p_k)_k$, say $q$ (which we can always suppose to exist by Arzel\`a--Ascoli and up to a subsequence), solves
\[
-\int_\Omega q \Delta \varphi + \int_{\partial \Omega} q \partial_n \varphi =0,  \qquad \forall \varphi\in C^2(\overline \Omega).
\] 
Since $q$ is average-free,  from \cref{L:uniq} below we get $q\equiv 0$, which contradicts \eqref{p_contin_est_smooth}, since 
\[
1=\|p_k\|_{C^{2\gamma}_*(\Omega)}\leq C\left( \|u_k\|^2_{C^{\gamma}(\Omega)}+\|p_k\|_{L^\infty(\Omega)}\right)\rightarrow 0.
\]
Thus, the validity of \eqref{p_contin_est_no_Linfty} follows.

\smallskip

\textit{Step 2}. 
Let $u$ be as in the statement of \cref{thm.main}. We wish to prove the existence of a unique solution $p$ such that \eqref{p_contin_est} and \eqref{p_contin_est_limit} hold true. Let $u_\eps$ be the regular approximation given by \cref{l:approx} and let $p^\eps$ the unique zero-average solution (which exists by \cref{t:main_approx}) of 
\begin{equation}\label{Neuman_problem_p_approx}
    \left\{
    \begin{array}{rcll}
        -\Delta p^\eps &=& \diver \diver (u_\eps\otimes u_\eps) & \text{ in }\Omega\\[1ex]
        \partial_n p^\eps&=&u_\eps\otimes u_\eps : \nabla n & \text{ on } \partial \Omega. 
    \end{array}
    \right.
\end{equation}
 By \eqref{p_contin_est_no_Linfty}  and \cref{l:approx}, we get 
\[
\|p^\eps\|_{C^{2\gamma}_*(\Omega)}\leq  C \|u_\eps\|^2_{C^\gamma(\Omega)}\leq C \|u\|^2_{C^\gamma(\Omega)}.
\]
Thus $(p^\eps)_\eps$ is a bounded sequence in $C^{2\gamma}_*(\Omega)$. In particular, by Arzel\`a--Ascoli Theorem, up to a (non-relabeled) subsequence, we can suppose that there exists $p\in C^{2\gamma}_*(\Omega)$ such that 
\[
\int_\Omega p\,dx=0,
\quad
p^\eps\rightarrow p\ \text{in}\ C^0(\overline \Omega)
\quad\text{and}\quad
\|p\|_{C^{2\gamma}_*(\Omega)}\leq C \|u\|^2_{C^\gamma(\Omega)}.
\]
For every $\eps>0$, we can test   \eqref{Neuman_problem_p_approx} with a sufficiently regular test function $\varphi$ and, by the regularity of $u_\eps$, we can run the same computations as in \eqref{reassembl_bound_cond}, getting that $p^\eps$ and $u_\eps$ satisfy \eqref{p_weaksol}. 
Since also $u_\eps \rightarrow u$ in $C^0(\overline \Omega)$, we can pass to the limit and deduce that the couple $(u,p)$ solves \eqref{p_weaksol}. To prove the uniqueness of $p$, suppose $p_1$ and $p_2$ are two zero-average solutions of \eqref{p_weaksol}. 
Then
\begin{equation*}
-\int_\Omega (p_1-p_2) \Delta \varphi + \int_{\partial \Omega} (p_1-p_2) \partial_n \varphi =0,  \qquad \forall \varphi\in C^2(\overline \Omega).
\end{equation*}
which, together with the constraint $\int_\Omega p_1=\int_\Omega p_2=0$, gives $p_1=p_2$ thanks to \cref{L:uniq}.
\end{proof}

\begin{lemma}[Uniqueness]\label{L:uniq}
Let $\delta\in (0,1]$ and $\Omega\subset \R^d$ a bounded domain with $\partial \Omega\in C^{2,\delta}$.  Then solutions $p\in C^0(\overline \Omega)$ to the problem \eqref{p_weaksol} are unique up to constants. 
\end{lemma}
Coherently with the regularity setting considered in the current paper, the uniqueness in the above lemma is stated in the class $p\in C^0(\overline \Omega)$. However, from the proof it will be clear that $p\in L^1(\Omega)$ suffices. Note that in such larger class one has to interpret the weak pressure formulation as in \eqref{p_very_very_weak}.
\begin{proof}
We have to prove that, if $q \in C^0(\overline \Omega)$ solves 
\begin{equation}\label{q_equal_zero}
-\int_\Omega q \Delta \varphi + \int_{\partial \Omega} q \partial_n \varphi =0,  \qquad \forall \varphi\in C^2(\overline \Omega),
\end{equation}
then necessarily $q$ is a constant. Since $-\int_\Omega \Delta \varphi + \int_{\partial \Omega}  \partial_n \varphi =0$, then $q-\frac{1}{|\Omega|}\int_\Omega q$ still solves \eqref{q_equal_zero}. Thus it is enough to show that, if $q$ is a $0$-average solution to \eqref{q_equal_zero}, then $q\equiv 0$.
First let $\varphi \in C^{2,\delta}(\Omega)$ be a solution to
$$
  \left\{
    \begin{array}{rcll}
        \Delta \varphi &=& 1 & \text{ in }\Omega\\[1ex]
        \partial_n \varphi &=&\frac{|\Omega|}{|\partial \Omega|} & \text{ on } \partial \Omega. 
    \end{array}
    \right.
$$
Note that the above problem admits a solution since the compatibility condition $\int_\Omega 1 = \int_{\partial \Omega} \frac{|\Omega|}{|\partial \Omega|}$ is satisfied. Plugging such test function in \eqref{q_equal_zero} yields
\begin{equation}
\label{boundary_average_zero}
0=\int_\Omega q=\frac{|\Omega|}{|\partial \Omega|}\int_{\partial\Omega} q.
\end{equation}
Now suppose that $q\not\equiv 0$. Extend $q$ trivially by $0$ outside $\Omega$ and consider $\text{sgn}\, q\in L^\infty(\R^d)$. Let $(\text{sgn}\, q)_\eps\in C^\infty (\overline \Omega)$ be its mollification. Since $q\not\equiv 0$ we can find a non-trivial $0$-average solution to 
$$
  \left\{
    \begin{array}{rcll}
        \Delta \varphi^\eps &=& (\text{sgn}\, q)_\eps & \text{ in }\Omega\\[1ex]
        \partial_n \varphi^\eps &=&c^\eps & \text{ on } \partial \Omega,
    \end{array}
    \right.
$$
where we have chosen $c^\eps= \frac{1}{|\partial \Omega|}\int_\Omega (\text{sgn}\, q)_\eps$ in order to enforce the compatibility condition. Clearly $\varphi^\eps \in C^{2,\delta}(\Omega)\subset C^2(\overline \Omega)$, thus \eqref{q_equal_zero} and \eqref{boundary_average_zero} imply $$
\int_\Omega q (\text{sgn}\, q)_\eps =c^\eps \int_{\partial \Omega} q=0\qquad \forall \eps>0.
$$
Letting $\eps\rightarrow 0$, we deduce $\int_\Omega |q|=0$, which yields a contradiction. 
\end{proof}

Thus, from now on, since the main goal is to prove \cref{t:main_approx}, we will always work with $u,p\in C^{1,\delta}(  \Omega)\cap C^\infty(\Omega)$ and thus directly with the formulation \eqref{eq.pressureEq} instead of the weak one \eqref{p_weaksol}, since in this regular setting  the two are equivalent. Moreover, note that the continuity estimate \eqref{p_contin_est_smooth} is the only thing that needs to be proven, since the existence (and the uniqueness) of the solution $p$ is a direct consequence of standard Elliptic Theory (see \cites{GT,V22} for instance).  Indeed, since $u\in C^{1,\delta}(  \Omega)$, we can equivalently rewrite \eqref{eq.pressureEq} as 
$$
    \left\{
    \begin{array}{rcll}
        -\Delta p &=& \diver \diver (u\otimes u) & \text{ in }\Omega\\[2mm]
        \partial_n p&=&- n\cdot\diver(u\otimes u) & \text{ on } \partial \Omega.
    \end{array}
    \right.
$$
In every $\Omega$ of class $C^2$ we have $ n\in C^{1}(\partial \Omega)$, and thus
\[
\diver \diver (u\otimes u)=\partial_i u^j\partial_j u^i\in C^{0,\delta}( \Omega),\quad n\cdot\diver (u\otimes u)  \in C^{0,\delta}(\partial\Omega),
\]
which, together with the compatibility condition 
\begin{align*}
\int_{\Omega} \diver \diver (u\otimes u)=\int_{\partial \Omega}n\cdot\diver (u\otimes u),
\end{align*}
ensure the solvability of \eqref{eq.pressureEq}, giving a pressure $p\in C^{1,\delta}(\Omega)$. Thus, from now on, we can forget about the existence and uniqueness statement and only focus on the estimate \eqref{p_contin_est_smooth}.
This, in turn, is a consequence of \cref{t:main_interior} and \cref{t:local_boundary} proved in Sections \ref{s:interior} and \ref{s:boundary}, respectively. The assumption $\partial \Omega \in C^{2,1}$ will be needed to prove \cref{t:local_boundary}, that is the boundary part of the estimate \eqref{p_contin_est_smooth}.

\section{Interior regularity}\label{s:interior}

In this section we wish to prove the following

\begin{theorem}[Interior regularity]
\label{t:main_interior}
Let $\gamma \in (0,\frac{1}{2}]$ and let $\Omega \subset \mathbb{R}^d$ be a bounded and simply connected domain of class $C^{2}$.  Let $u \in C^{1,\delta}(\Omega)\cap C^\infty (\Omega)$,  for some $\delta\in (0,1)$, be a divergence-free vector field such that $u\cdot n|_{\partial\Omega}=0$. Then, for every $\widetilde \Omega \Subset\Omega$, the unique zero-average solution $p\in C^{1,\delta}(\Omega)$ of~\eqref{eq.pressureEq} enjoys
\begin{equation}\label{p_contin_est_smooth_interior}
    \|p\|_{C^{2\gamma}_*\left(\widetilde \Omega\right)} \leq C \left( \|u\|_{C^{\gamma}(\Omega)}^2 +\|p\|_{L^\infty(\Omega)}\right),
\end{equation}
for some constant $C>0$ depending on $\Omega$, $\widetilde \Omega$ and $\gamma$ only.
\end{theorem}
The previous theorem being a purely interior regularity result, in the case $\gamma<\frac{1}{2}$ it does not contain anything new with respect to the double regularity established in \cites{Is2013,CD18} in the whole space, since one can simply localize the equation \eqref{eq.pressureEq} strictly inside $\Omega$ and then extending it to $\R^d$. However, since \cref{t:main_interior} additionally provides the new estimate  in the borderline case $\gamma=\frac{1}{2}$ (which was left open in \cites{Is2013,CD18}), we are going to give a detailed proof in the whole range $\gamma\leq \frac{1}{2}$, noticing that the borderline case $\gamma=\frac12$ does not require any different argument with respect to case $\gamma <\frac12$.

\begin{proof}
Let $u$ be as in the statement. Extend it to the whole space (see for instance \cite{KMPT00}*{Section~5}), getting a divergence-free vector field $\tilde u\in C^1_c(\R^d)$ such that 
\begin{equation}
    \label{est_extension_cont}
    \|\tilde u\|_{C^\gamma(\R^d)}\leq C\|u\|_{C^\gamma(\Omega)}.
\end{equation}
Define $q$ to be the unique bounded solution, decaying at infinity, of 
\[
-\Delta q=\diver \diver (\tilde u\otimes \tilde u)\quad \text{in } \R^d.
\]
Then $p-q$ is harmonic in $\Omega$ and thus we can estimate it, for every $\gamma \leq \frac12$ and every $\widetilde \Omega \Subset\Omega$, as 
\[
\|p-q\|_{C^{2\gamma}_*(\widetilde \Omega)}\leq C\|p-q\|_{C^{1}(\widetilde\Omega)}\leq C \|p-q\|_{L^\infty(\Omega)}\leq C \left( \|p\|_{L^\infty(\Omega)}+\|q\|_{L^\infty(\R^d)}\right),
\]
where in the second inequality we have used that $\|\nabla f\|_{L^\infty(\tilde \Omega)}\lesssim \|f\|_{L^\infty(\Omega)}$, if $f$ is harmonic in~$\Omega$.
Thus we infer that 
\[
\|p\|_{C^{2\gamma}_*(\widetilde \Omega)}\leq \|p-q\|_{C^{2\gamma}_*(\widetilde \Omega)}+\|q\|_{C^{2\gamma}_*(\widetilde \Omega)}\leq C \left( \|p\|_{L^\infty(\Omega)}+\|q\|_{C^{2\gamma}_*(\R^d)}\right).
\]
Moreover, by \cref{t:double_Rd_whole_range} below, and also using \eqref{est_extension_cont}, we have
\begin{equation*}
\|q\|_{C^{2\gamma}_*(\R^d)}\leq C \|\tilde u\|^2_{C^\gamma(\R^d)}\leq C \| u\|^2_{C^\gamma(\Omega)}
\end{equation*}
for some constant $C>0$ depending on $\gamma$. This concludes the proof.
\end{proof}

The following theorem provides the full double regularity of the pressure in the whole space, for all $\gamma\leq\frac12$, thus both extending the results \cites{Is2013,CD18} to the borderline case $\gamma=\frac{1}{2}$ and also improving to $C^1_*(\R^d)$ the $\loglip(\R^d)$ regularity obtained in \cite{C2014}. By straightforward modifications, the very same result holds in the periodic setting of the torus $\T^d$.

\begin{theorem}[Double regularity on $\R^d$]\label{t:double_Rd_whole_range}
Let $\gamma\in (0,\frac12]$ and $v\in C_c^\gamma(\R^d)$ be a compactly supported divergence-free vector field. Then, there exists a unique solution decaying at infinity of
\[
-\Delta q=\diver \diver (v\otimes v)\quad \text{in } \R^d
\]
satisfying
\begin{equation}\label{q_double_whole_space_whole_range}
\|q\|_{C^{2\gamma}_*(\R^d)}\leq C \|v\|^2_{C^\gamma(\R^d)}.
\end{equation}
\end{theorem}
\begin{proof}
Uniqueness is trivial. To prove the existence, one can simply regularize the vector field $v$ as $v_\eps\in C^\infty_c(\R^d)$ and then look for the corresponding solution $q^\eps$. By smoothness of the right-hand side, there exists $q^\eps\in C^\infty(\R^d)$ solving $-\Delta q^\eps=\diver \diver (v_\eps\otimes v_\eps)$. Since, for all $\eps>0$,
\[
\|v_\eps\|_{C^\gamma(\R^d)} \leq \|v\|_{C^\gamma(\R^d)},
\]
it is then enough to prove the continuity estimate \eqref{q_double_whole_space_whole_range} for $q^\eps$ and $v_\eps$, from which one can then obtain a solution $q$ by Arzel\`a--Ascoli Theorem, up to a subsequence. To lighten the notation we  simply write $q$ and $v$ in place of $q^\eps$ and $v_\eps$.

\jump

The global double regularity \eqref{q_double_whole_space_whole_range} follows by a Littlewood--Paley analysis. We believe that the following  computations also give a clearer intuition on why one should expect that the very same double regularity can still be derived in our more general case of \eqref{eq.pressureRd}. We refer the reader to \cite{Is2013}, where the Littlewood--Paley formalism has been indeed used to prove the double pressure regularity for the first time, for every $0<\gamma<\frac12$.

\jump

Let us write $q = \mathbf{P}_1q + \mathbf{P}_{>1}q$, where $\mathbf{P}_{1}$ stands for the first Littlewood--Paley smooth truncation on frequencies $|\xi| \leq 1$. We refer to \cref{s:function_spaces} for the notation and the basic facts on Littlewood--Paley calculus. The reason for distinguishing between high and low frequencies is that the high-frequency part of the Laplacian is invertible. More precisely, recall that $\mathbf{P}_{> 1}(\xi)=1-\chi(\xi)$ for some smooth compactly function $\chi$ which equals $1$ in a neighbourhood of $0$.
The symbol 
\[
    a(\xi)=-\frac{1-\chi(\xi)}{|\xi|^2}
\]
belongs to the class $S^{-2}$ and therefore $A_{-2} \coloneqq \operatorname{Op}(a(\xi)) \in \operatorname{Op}(S^{-2})$ is a good candidate for the inverse of $-\Delta$ on high frequencies. We indeed have $A_{-2} \circ (-\Delta)=\mathbf{P}_{> 1}$, this equality being exact because the two operators are Fourier multipliers. Therefore, we can write
\[
    \mathbf{P}_{> 1}q = A_{-2}\partial_{ij}^2(v^iv^j). 
\]
Let us use a Littlewood--Paley decomposition for $v$ as 
\[
   v^i = \sum_{M \in 2^{\mathbb{N}}} v_M^i,
\]
so that we can write 
\[
   v^iv^j = \sum_{M\sim K} v^i_Mv^j_K + \sum_{M \ll K} (v_{M}^iv_K^j +v_{K}^iv_M^j).
\]
Since the $v_{N}^i$ are smooth and divergence-free (note that the Littlewood--Paley truncation $\mathbf{P}_N$ commutes with the divergence, as they are both Fourier multipliers), we have 
\[
    \partial_{ij}^2(v^i_Kv^j_M) = \partial_jv^i_K\partial_iv^j_M,
\]
which we use on the $\sum_{M \ll K}$ terms to obtain 
\[
    \mathbf{P}_{> 1}q = \sum_{M\sim K} A_{-2}\partial_{ij}^2(v^i_Mv^j_K) + \sum_{M\ll K} A_{-2}\left(\partial_jv^i_K\partial_iv^j_M + \partial_jv^i_M\partial_iv^j_K\right). 
\]
Finally, since $\mathbf{P}_N$ is a Fourier multiplier, and so commutes with $A_{-2}$ and $\partial^2_{ij}$, for $N\geq 2$ we have
\begin{align*}
    q_N&=\sum_{M\sim K} A_{-2}\partial_{ij}^2\mathbf{P}_N(v^i_Mv^j_K) + \sum_{M\ll K} A_{-2}\mathbf{P}_N\left(\partial_jv^i_K\partial_iv^j_M + \partial_jv^i_M\partial_iv^j_K\right)=: I_N + J_N.
\end{align*}

\jump

\noindent\emph{Estimation of $I_N$.} First, let us remark that $v_M^iv_K^j$ is frequency supported in $|\xi| \leq \max\{M,K\}$, therefore $\mathbf{P}_N(v_M^iv_K^j)=0$ unless $\max\{M,K\} \gtrsim N$. We thus have 
\[
I_N = \sum_{K\sim M \gtrsim N} A_{-2}\partial_{ij}^2\mathbf{P}_N(v^i_Mv^j_K). 
\]
Now, $A_{-2}\partial_{ij}^2 = \operatorname{Op}\left(-\frac{(1-\chi(\xi))\xi_i\xi_j}{|\xi|^2}\right) \in \operatorname{Op}(S^0)$ is a $C^0_* \to C^0_*$ continuous $0$-order operator, according to \cref{th.calderonVaillancourt}.
Thus,  
\[
    \|I_N\|_{C^{0}_*} \lesssim \sum_{M\sim K \gtrsim N} \left\|\mathbf{P}_N(v_M^iv_K^j)\right\|_{C^0_*} \lesssim \sum_{M\sim K \gtrsim N} \left\|v_M^iv_K^j\right\|_{L^{\infty}},  
\]
where we used the definition of the $C^0_*$ norm in the last inequality. We now use that $v \in C^{\gamma}_*(\R^d)$, so that $\|v_K^j\|_{L^\infty} \lesssim K^{-\gamma}\|v\|_{C_*^{\gamma}}$, and therefore 
\begin{equation*}
    \|I_N\|_{C^{0}_*} \lesssim \sum_{M\sim K \gtrsim N} M^{-\gamma}K^{-\gamma} \|v\|_{C_*^{\gamma}}^2\lesssim \sum_{M \gtrsim N} M^{-2\gamma} \|v\|_{C_*^{\gamma}}^2 \lesssim N^{-2\gamma}\|v\|_{C_*^{\gamma}}^2.   
\end{equation*}
In particular,  $I_N$ being frequency localized at $N$, and using the definition of the $C^0_*$ norm, the previous estimate implies 
\begin{equation}\label{eq:estI_infty}
    \|I_N\|_{L^\infty(\R^d)} \lesssim N^{-2\gamma}\|v\|_{C_*^{\gamma}}^2.   
\end{equation}
Notice that, in the estimate of this term, we have been able to double the regularity, because the frequencies of the two $v$'s were similar, therefore responsible for the addition of regularities.

\jump

\noindent\emph{Estimation of $J_N$.} Note that $\partial_jv^i_K\partial_iv^j_M$ is supported in $\{|\xi| \sim K\}$ because $M\ll K$, and therefore $\mathbf{P}_N(\partial_jv^i_K\partial_iv^j_M)=0$ unless $N\sim K$. We combine this remark and the continuity of $A_{-2} \colon C^{-2}_* \to C^0_*$ (because $A_{-2}$ is a pseudodifferential operator of order $-2$) to obtain 
\begin{align*}
    \|J_N\|_{C_*^0} & \lesssim \sum_{M\ll K \sim N} \left\|\mathbf{P}_N\left(\partial_jv^i_K\partial_iv^j_M + \partial_jv^i_M\partial_iv^j_K\right)\right\|_{C^{-2}_{*}} \\
    & \lesssim N^{-2}\sum_{M\ll K \sim N} \left\|\partial_jv^i_K\partial_iv^j_M + \partial_jv^i_M\partial_iv^j_K\right\|_{L^{\infty}}. 
\end{align*}
By Bernstein's inequality (see (ii) in \cref{thm.bernstein}) and the fact that $v\in C^{\gamma}_*(\R^d)$, we see that $\|\partial_iv^j_K\|_{L^{\infty}(\R^d)} \lesssim K^{1-\gamma}\|v\|_{C^{\gamma}_*(\R^d)}$ so that 
\begin{equation*}
    \|J_N\|_{C_*^0(\R^d)} \lesssim N^{-2}\sum_{M\ll K \sim N}M^{1-\gamma}K^{1-\gamma}\|v\|^2_{C^{\gamma}_*(\R^d)} \lesssim N^{-2\gamma} \|v\|^2_{C^{\gamma}_*(\R^d)}.
\end{equation*}
In particular, as before, this implies 
\begin{equation}\label{eq:estII}
    \|J_N\|_{L^\infty(\R^d)} \lesssim  N^{-2\gamma} \|v\|^2_{C^{\gamma}_*(\R^d)}.
\end{equation}
Notice that, in the above estimate, we gained double regularity because the double divergence has been rewritten as a product of gradients. Without this rewriting, since $K\gg M$, all derivatives should fall on the function localized at $\sim K$, but with the rewriting we transfer one derivative to the lowest frequency $M$, and this derivative is estimated by $M$ (instead of $K$), hence the gain.

\jump

Combining~\eqref{eq:estI_infty} and~\eqref{eq:estII} we finally obtain that, for all $N\geq 2$, there holds
\begin{equation}\label{IandJ_est}
    N^{2\gamma}\|q_N\|_{L^{\infty}} \lesssim \|v\|^2_{C^{\gamma}_*}=\|v\|^2_{C^{\gamma}},
\end{equation}
since $\gamma\not \in \mathbb N$.
Recall that $q=\mathbf{P}_1 q+\mathbf{P}_{>1} q$, so we need to estimate $\mathbf{P}_1 q$. By Sobolev embedding, if $s>\frac{d}{2}$, then 
\begin{align*}
    \|\mathbf{P}_1 q\|^2_{L^\infty(\R^d)}&\lesssim  \||\nabla|^s\mathbf{P}_1 q\|^2_{L^2(\R^d)}=\int_{\R^d} \left|\widehat{\mathbf{P}_1 q} (\xi)\right|^2 |\xi|^{2s} \chi_{|\xi|\leq 1}(\xi)\,d\xi\\
    &\lesssim \int_{\R^d} \left|\widehat{\mathbf{P}_1 q} (\xi)\right|^2 \chi_{|\xi|\leq 1}(\xi)\,d\xi\lesssim \int_{\R^d} \left|\widehat{v^i v^j}(\xi)\right|^2  \chi_{|\xi|\leq 1}(\xi)\,d\xi\\
    &\lesssim \|v^iv^j\|^2_{L^2(\R^d)}\lesssim \| v\|^4_{C^0(\R^d)},
\end{align*}
where in the fourth inequality we have used that $-|\xi|^2q(\xi)=\xi_i\xi_j \widehat{v^iv^j}(\xi)$, for every $\xi\in \R^d$, while in the last inequality we have used the compact support of $v$. This, together with \eqref{IandJ_est}, gives \eqref{q_double_whole_space_whole_range}.
The proof is thus complete.
\end{proof}

\section{Boundary regularity, part 1: extending to the whole space}\label{s:boundary}

The aim of this section, and of the subsequent \cref{s:main_proof_new}, is to provide a proof of the following

\begin{theorem}[Local boundary regularity]\label{t:local_boundary}
Let $\gamma \in (0,\frac{1}{2}]$, $\delta \in (0,1)$ and  $\Omega \subset \mathbb{R}^d$ be a bounded and simply connected domain of class $C^{2,1}$.  Let $u \in C^{1,\delta}(\Omega)\cap C^\infty(\Omega)$ be a divergence-free vector field such that $u\cdot n|_{\partial\Omega}=0$ and let $x_0\in \partial \Omega$. Then, there exists a ball $B_{R_0}(x_0)$ such that the unique zero-average solution $p\in C^{1,\delta}(\Omega)$ of~\eqref{eq.pressureEq} enjoys
\begin{equation}\label{p_contin_est_boundary_local}
    \|p\|_{C^{2\gamma}_*\left(\Omega\cap B_{R_0}(x_0) \right)} \leq C \left(\|u\|_{C^{\gamma}(\Omega)}^2+\|p\|_{L^\infty(\Omega)}\right) 
\end{equation}
for some constant $C>0$ depending on $\gamma$, $R_0$ and $\Omega$ only.
\end{theorem}
Note that, by compactness, we can always cover $\partial \Omega$ with a finite number of balls.
This immediately implies that, by patching all the local estimates \eqref{p_contin_est_boundary_local},
\[
\|p\|_{C^{2\gamma}_*\left(\Omega\setminus \Omega_{\tilde R}\right)} \leq C\left( \|u\|_{C^{\gamma}(\Omega)}^2+\|p\|_{L^\infty(\Omega)}\right),
\]
for a (possibly larger) constant $C>0$ and some $\tilde R>0$, where we defined
\[
\Omega_{\tilde R}\coloneqq\left\{x\in \Omega \, : \, \text{dist }(x,\partial \Omega)\geq \tilde R \right\}.
\]
This, together with the interior regularity result of \cref{t:main_interior}, proves \cref{t:main_approx}, from which we already deduced \cref{thm.main}.

\jump

Thus, from now on, we only focus on \eqref{p_contin_est_boundary_local}. To prove such a quantitative local estimate, we use the normal geodesic coordinates introduced in  \cref{prop.geometry}, getting a modified equation in the new variables. We  then extend such equation to the whole space $\R^d$ and apply pseudodifferential calculus, together with Littlewood--Paley analysis, in order to prove the desired regularity. In this \cref{s:boundary} we focus on the first step, \textit{i.e.}, changing coordinates and extending to the whole space. Then, in \cref{s:main_proof_new}, we prove the quantitative estimate \eqref{p_contin_est_boundary_local}. 

\subsection{Removing the boundary datum} 

Here we prove that the boundary datum is not the main obstacle for proving the double regularity of the pressure, but it only represents a compatibility condition in order to guarantee that \eqref{eq.pressureEq} admits a solution. 
Removing the boundary datum is then useful in order to pass to the normal geodesic coordinates, avoiding some extra (tedious) terms at the boundary.

Note that, in the result below, $u$ does not need to be divergence-free or tangent to~$\partial \Omega$.

\begin{lemma}\label{l:remove_boundary} Let $\Omega \subset \R^d$ be of class $C^{2,\delta}$, for some $\delta\in (0,1)$. Let $\gamma\in (0,1)$ and $u\in C^\gamma(\Omega)$ be any vector field. Denote $\beta=\min\{\delta,\gamma\}$ and $A=\frac{1}{|\Omega|}\int_{\partial\Omega} u\otimes u : \nabla n$. Then, there exists a unique zero-average $\psi \in C^{1,\beta}(\Omega)$ such that 
\[
  \left\{
    \begin{array}{rcll}
        \Delta \psi &=& A & \text{ in }\Omega\\[1ex]
        \partial_n \psi &=&u\otimes u : \nabla n & \text{ on } \partial \Omega,
    \end{array}
    \right.
\]
and, moreover, there exists a constant $C>0$ such that
\begin{equation}\label{psi_regular}
\|\psi\|_{C^{1,\beta}(\Omega)}\leq C \|u\|^2_{C^\gamma(\Omega)}.
\end{equation}
\end{lemma}

\begin{proof} 
Clearly, $A\in \R$ and $u\otimes u : \nabla n\in C^\beta(\partial \Omega)$. Moreover, the compatibility condition
\[
\int_{\partial \Omega} u\otimes u:\nabla n =\int_{\partial \Omega} \partial_n \psi =\int_{\Omega} \Delta \psi =|\Omega| A
\]
holds true. 
Thus, by standard elliptic arguments (see \cite{V22}*{Theorem 1.2} for instance), we infer that there exists a unique $\psi\in C^{1,\beta}(\Omega)$ such that $\int_{\Omega} \psi=0$ and 
\[
\|\psi \|_{C^{1,\beta}(\Omega)}\leq C\left(|A|+ \| u\otimes u : \nabla n\|_{C^\beta(\partial \Omega)}+\|\psi\|_{L^\infty(\Omega)} \right).
\]
Since
\begin{align*}
     |A|\leq C \| u\|^2_{C^0(\overline \Omega)} \qquad \text{and} \qquad  \| u\otimes u : \nabla n\|_{C^\beta(\partial \Omega)}\leq C\| u\|^2_{C^\gamma(\Omega)},
\end{align*}
for some constant $C>0$ depending on $\Omega$ only, we achieve 
\[
\|\psi\|_{C^{1,\beta}(\Omega)}\leq C \left( \|u\|^2_{C^\gamma(\Omega)}+\|\psi\|_{L^\infty(\Omega)} \right).
\]
The desired estimate \eqref{psi_regular} then follows by eliminating the term $\|\psi\|_{L^\infty(\Omega)}$ from the right-hand side of the previous estimate. This can be done by exploiting the very same contradiction argument already used in the proof of \cref{thm.main} above.
We omit the details.
\end{proof}

Up to changing $p$ with $p-\psi$, we are left with proving double regularity for solutions of 
\begin{equation}
    \label{eq.pressureEqNoBoundary}
    \left\{
    \begin{array}{rcll}
        -\Delta p &=& \diver \diver (u\otimes u) + A & \text{ in }\Omega\\[1ex]
        \partial_n p&=&0 & \text{ on } \partial \Omega,
    \end{array}
    \right.
\end{equation}
for some constant $A\in\R$ guaranteeing the compatibility condition for solving \eqref{eq.pressureEqNoBoundary}. More precisely, \cref{t:local_boundary} is a consequence of the following 

\begin{theorem}
\label{t:main_boundary_local_noboundary}
Let $\gamma \in (0,\frac{1}{2}]$, $\delta \in (0,1)$ and  $\Omega \subset \mathbb{R}^d$ be a bounded and simply connected domain of class $C^{2,1}$.  Let $u \in C^{1,\delta}( \Omega)\cap C^\infty(\Omega)$ be a divergence-free vector field such that $u\cdot n|_{\partial\Omega}=0$ and let $x_0\in \partial \Omega$. Then, there exists a ball $B_{R_0}(x_0)$ such that the unique zero-average solution $p\in C^{2,\delta}(\Omega)$ of~\eqref{eq.pressureEqNoBoundary}, with $A=\frac{1}{|\Omega|}\int_{\partial \Omega} u\otimes u :\nabla n$, enjoys
\begin{equation}\label{p_contin_est_boundary_local_noboundary}
    \|p\|_{C^{2\gamma}_*\left(\Omega\cap B_{R_0}(x_0) \right)} \leq C \left(\|u\|_{C^{\gamma}(\Omega)}^2+\|p\|_{L^\infty(\Omega)}\right) .
\end{equation}
for some constant $C>0$ depending on $\gamma$, $R_0$ and $\Omega$ only.
\end{theorem}

Note that, since $u\in C^1(\overline \Omega)\cap C^\infty(\Omega)$ is divergence-free and tangent to the boundary (which in particular gives that $\partial \Omega$ is a level set of the scalar function $u\cdot n$), we have
\begin{align*}
\int_{\Omega} \diver \diver (u\otimes u)&=\int_{\partial \Omega} \partial_j(u^i u^j)n^i=\int_{\partial \Omega} u^j \partial_j u^i n^i =\int_{\partial \Omega} u^j \partial_j(u\cdot n)-\int_{\partial \Omega} u^iu^j \partial_j n^i\\
&=-\int_{\partial \Omega} u\otimes u :\nabla n=-|\Omega| A.
\end{align*}
Therefore, the compatibility condition in the Neumann boundary value problem \eqref{eq.pressureEqNoBoundary} is satisfied and the existence of a solution, say $p\in C^{2,\delta}(\Omega)$, unique up to constants, follows by standard Elliptic Theory. Thus we only have to prove the estimate \eqref{p_contin_est_boundary_local_noboundary}.

\jump

Even if straightforward, we give a detailed proof on how the previous result implies \cref{t:local_boundary} for the reader's convenience.

\begin{proof}[Proof of \cref{t:local_boundary}]
Let $u$ be as in the statement of \cref{t:local_boundary}. Let $A=\frac{1}{|\Omega|}\int_{\partial\Omega} u\otimes u : \nabla n$ and let $\psi$ be the corresponding unique solution of 
\[
  \left\{
    \begin{array}{rcll}
        \Delta \psi &=& A & \text{ in }\Omega\\[1ex]
        \partial_n \psi &=&u\otimes u : \nabla n & \text{ on } \partial \Omega
    \end{array}
    \right.
\]
given by \cref{l:remove_boundary}, with $\int_\Omega \psi=0$.

\jump

Let $q$ be the unique zero-average solution of \eqref{eq.pressureEqNoBoundary} given by \cref{t:main_boundary_local_noboundary}. Then $p=q+\psi$ solves \eqref{eq.pressureEq} and it is also average-free (since both $q$ and $\psi$ are), which in particular implies its uniqueness. Moreover, since $2\gamma\leq 1$, by \eqref{psi_regular} and \eqref{p_contin_est_boundary_local_noboundary} we get
\begin{align*}
\|p\|_{C^{2\gamma}_*\left(\Omega\cap B_{R_0}(x_0)\right)}&\leq \|q\|_{C^{2\gamma}_*\left(\Omega\cap B_{R_0}(x_0)\right)}+\|\psi\|_{C^{2\gamma}_*\left(\Omega\cap B_{R_0}(x_0)\right)}\\
&\leq 
C\left( \|u\|^2_{C^{\gamma}(\Omega)}+\|q\|_{L^\infty(\Omega)}+ \|\psi\|_{C^1\left(\overline\Omega\right)}\right)\\
&\leq C \left( \|u\|^2_{C^{\gamma}(\Omega)}+\|p\|_{L^\infty(\Omega)}+ \|\psi\|_{C^1\left(\overline\Omega\right)}\right)\\
&\leq C \left( \|u\|^2_{C^{\gamma}(\Omega)}+\|p\|_{L^\infty(\Omega)} \right),
\end{align*}
where in the third inequality we also used  $\|q\|_{L^\infty(\Omega)}\leq \|p\|_{L^\infty(\Omega)}+\|\psi\|_{L^\infty(\Omega)}$.
\end{proof}

Thus the rest of the paper is devoted to the proof of \cref{t:main_boundary_local_noboundary}, that is, as already pointed out, to the proof of the estimate \eqref{p_contin_est_boundary_local_noboundary}.

\subsection{Straightening of the boundary} 

We now reduce the proof of~\eqref{p_contin_est_boundary_local_noboundary} to the flat-boundary case by straightening the boundary using the results of~\cref{sec.geometry}. 

Let $U$ be a neighborhood of  $x_0\in \partial \Omega$. We can assume $U$ to be small enough so that, in $U\cap \Omega$, we have local geodesic coordinates $(r,\theta^2, \dots, \theta^d)=(r,\theta) \in (-r_0,r_0) \times \Theta$ such that the metric $g^{ij}$ in these coordinates satisfies the properties of \cref{prop.geometry}.
Therefore~\eqref{eq.pressureEqNoBoundary} (denoting the solution by $p$, thus slightly abusing notation) now reads as
\[
     -\frac{1}{G} \partial_i\left(G g^{ij}\partial_j p\right) = \frac{1}{G} \partial^{2}_{ij}\left(G u^iu^j\right)+A \quad \text{ in }\Omega,
\]
with the boundary condition 
\[
    \partial_r p(0,\theta)=0\quad \text{ for all } \theta \in \Theta,
\]
where $G=G(r,\theta)=\sqrt{\det g(r,\theta)} \in C^{0,\delta}$.
Pay attention: here $u^i, u^j$ refer to $u^r$ and $u^{\theta^i}$, as well as the partial derivatives $\partial_i$ stands for $\partial_r$ and $\partial_{\theta^i}$. Moreover, the condition $u\cdot n = 0$ on $U \cap \partial\Omega$ reads $u^r(0,\theta)=0$. 
Multiplying the equation in the interior by $G$ leads to
\begin{equation}
    \label{eq.pressureGeodesic}
    \left\{
    \begin{array}{rcll}
        -\partial_i\left(G g^{ij}\partial_j p\right) &=&\partial^{2}_{ij}\left(G u^iu^j\right)+GA  & \text{ in }\Omega\cap U\\[1ex]
        \partial_r p(0,\theta)&=&0 & \text{ for all } \theta \in \Theta. 
    \end{array}
    \right.
\end{equation}

\begin{remark}\label{r:constant_neglect}
Note that the function $GA$ plays a purely compatibility role only. Indeed,  on a $C^{2,1}$ domain
\[
\|GA\|_{C^{0,1}(\Omega\cap U)}\leq |A|\|G\|_{C^{0,1}(\Omega\cap U)}\leq C(\Omega)\|u\|^2_{C^0(\overline \Omega)},
\]
which is a term that can be easily incorporated in all the estimates in \cref{s:main_proof_new} below, where we prove the double regularity of the `modified' pressure equation in the local geodesic coordinates. Thus, from now on, we can forget about such a very regular term and assume that $A=0$.
\end{remark}

\subsection{Extension for negative \texorpdfstring{$r$}{r}} 

Now we extend the functions for $r<0$ as follows: for any $r>0$ we let $\tilde{g}(-r,\theta)\coloneqq g(r,\theta)$, $\tilde{u}^r(-r,\theta)\coloneqq -u^r(r,\theta)$, $\tilde{u}^{\theta}(-r,\theta)\coloneqq u^{\theta}(r,\theta)$ and $\tilde{p}(-r,\theta)\coloneqq p(r,\theta)$. 

\begin{remark} Note that we evenly extend $g$ and $u^{\theta}$ because we have no information about their value at $r=0$. On the other hand, we exploit the fact that $u^r(0,\theta)=0$ by extending $u^r$ oddly. Similarly, since $\partial_rp(0,\theta)=0$, we want $\partial_rp$ to be extended oddly, so $p$ should be extended evenly.
\end{remark} 

Thanks to the boundary condition satisfied by $u$, we have the following 

\begin{lemma}[Regularity of the extensions]\label{l:regular_extension} Let $\delta \in (0,1]$ and $\Omega$ a $C^{2,\delta}$ bounded domain. 
Outside $\{r=0\}$ all the functions above are $C^{1,\delta}$.  Moreover we have the following global regularities:
\begin{enumerate}
    \item $ \tilde{p},\tilde u\in \Lip_{r,\theta}$ and $\partial_r \tilde u^r\in C^0_{r,\theta}$;
    \item $\tilde{g}^{\theta\theta}\in C^{0,\delta}_{r,\theta}$;
    \item $\partial_{\theta ^i}\tilde{u}^r$, $\partial_{\theta^i}\tilde{u}^{\theta}$ and $\partial_{\theta^i}\tilde{p}$ are $C^0_{r,\theta}$ functions.
\end{enumerate}
\end{lemma}

\begin{proof} 
Since $u^r(0,\theta)=0$ and $u\in C^1(\overline \Omega)$, then clearly $\tilde u\in \Lip_{r,\theta}$ and $\partial_r\tilde u^r\in C^0_{r,\theta}$. The same reasoning applies to $\tilde p$, giving $(1)$. Since $\Omega$ is a $C^{2,\delta}$ domain, $g\in C^{0,\delta}$, thus all the $\theta$ components of its even extension satisfy $\tilde{g}^{\theta\theta}\in C^{0,\delta}_{r,\theta}$, which gives $(2)$. 
Moreover, since we are evenly (or oddly for $u^r$) extending  in the $r$ variable only, it is also clear that also $(3)$ holds true, giving the $C^1$ regularity of $\tilde u$ and $\tilde p$ in the $\theta$ variables. 
\end{proof}

Consequently, whenever $\partial \Omega \in C^{2,1}$, we have that $\tilde{g}^{\theta\theta}\in \Lip_{r,\theta}$. In particular $\tilde{G}\coloneqq\sqrt{\det \tilde{g}}$ is even in $r$ and globally Lipschitz.
\jump

We now have to make sure that the extension process did not create distributional jumps at the PDE level, that is $\tilde p$ and $\tilde u$ solve, on the whole space, the same equation as for $r>0$.  This will be a consequence of the following lemma, in which we will keep the variables notation $r,\theta$ as introduced above. 
\begin{lemma}\label{L:divergence_cont_vector}
Suppose $V\in C^0(\R^d)$ satisfies $\diver V=0$ in $\mathcal D'(\{r\neq 0\})$. Then $\diver V=0$ in $\mathcal D'(\R^d)$.
\end{lemma}
\begin{proof}
Let $\varphi\in C^\infty_c(\R^d)$. Since $\diver V=0$ in $\mathcal D'(\{r\neq 0\})$ and $V$ is continuous,  for every $\eps>0$ we have 
$$
\int_{\{-\eps<r<\eps\}^c} V\cdot \nabla \varphi=\int_{\{r=\eps\}} V\cdot n \varphi - \int_{\{r=-\eps\}} V\cdot n \varphi,
$$
where $n=(1,0,\dots,0)$. Thus
\begin{align*}
\langle \diver V,\varphi \rangle &=-\int V\cdot \nabla \varphi =-\lim_{\eps\rightarrow 0} \int_{\{-\eps<r<\eps\}^c} V\cdot \nabla \varphi\\
&=\lim_{\eps\rightarrow 0} \left( \int_{\{r=\eps\}} V\cdot n \varphi - \int_{\{r=-\eps\}} V\cdot n \varphi \right).
\end{align*}
The last limit vanishes since $V\in C^0(\R^d)$, from which we conclude that $\diver V=0$ in $\mathcal D'(\R^d)$ by the arbitrariness of $\varphi$.
\end{proof}

Then we have the following

\begin{lemma}[Divergence of the extension]\label{l:extension_divfree} In a $C^{2,\delta}$ domain $\Omega$, $\delta \in (0,1]$, such extension of $u$ preserves the incompressibility, \textit{i.e.}, $\diver \tilde{u}=0$, where the divergence is taken in the $\tilde{g}$ metric. 
\end{lemma}

\begin{proof} First, let us compute the divergence at $(r,\theta)$ for $r>0$.
In this case, there holds $\diver \tilde u = \frac{1}{\tilde G} \partial_i (\tilde G \tilde{u}^i) = \diver u(r,\theta)=0$. Similarly, for $r<0$, the parity properties of $u$ give the same result. By \cref{l:regular_extension} we have $\tilde G \tilde u\in C^0$, and  \cref{L:divergence_cont_vector} concludes the proof. 
\end{proof}

Similarly, we need to make sure that distributional jumps do not appear in the elliptic PDE. Since the latter can be rewritten as 
$$
\partial_i \left( \tilde{G}\tilde{g}^{ij}\partial_j\tilde{p} -\partial_{j}(\tilde{G}\tilde{u}^i\tilde{u}^j) \right)=0\qquad \text{in } \mathcal D'(\{r\neq 0\}),
$$
again by \cref{L:divergence_cont_vector}, it is enough to check that both $\tilde{G}\tilde{g}^{ij}\partial_j\tilde{p}$ and  $\partial_{j}(\tilde{G}\tilde{u}^i\tilde{u}^j)$ are continuous, for all $i=1,\dots,d$.

\begin{lemma}[Laplacian of the extension] Let $\Omega$ a $C^{2,\delta}$ domain, for some $\delta \in (0,1]$. For all $i=1,\dots,d$ it holds $\tilde{G}\tilde{g}^{ij}\partial_j\tilde{p}\in C^0_{r,\theta}$. 
\end{lemma}
\begin{proof} Clearly $\tilde G \tilde g^{ij}\in C^0_{r,\theta}$ by \cref{l:regular_extension}, thus it is enough to check that $\partial_j \tilde p\in C^0_{r,\theta}$ for all $j=1,\dots,d$. By \cref{l:regular_extension} we have $\partial_{\theta^j}\tilde p \in C^0_{r,\theta}$, for all $j\neq 1$. Moreover, since  $\partial_r \tilde{p}$ is odd in $r$ and $\partial_r p(0,\theta)=0$, we also that $\partial_r\tilde p\in C^0_{r,\theta}$.
\end{proof}

We just need to exclude the possible jumps for the double divergence term.

\begin{lemma}[Double divergence of the extension]
\label{l:dd ext}
Let $\Omega$ a $C^{2,1}$ domain.  For all $i=1,\dots,d$ it holds  $\partial_{j}(\tilde{G}\tilde{u}^i\tilde{u}^j)\in C^0_{r,\theta}$. 
\end{lemma}

\begin{proof} First, note that $\tilde{G}\tilde{u}^i\tilde{u}^j$ is a Lipschitz function, therefore it has no jump across $\{r=0\}$.
Moreover, by Rademacher's Theorem, the distributional derivative $\partial_j(\tilde{G}\tilde{u}^i\tilde{u}^j)$ agrees with its pointwise computation
\[
    \partial_{j}(\tilde{G}\tilde{u}^i\tilde{u}^j) = \partial_j(\tilde{G}\tilde{u}^j)\tilde{u}^i + \tilde{G}\tilde{u}^j \partial_j \tilde{u}^i = \tilde{G}\tilde{u}^j \partial_j \tilde{u}^i,
\]
where we used the divergence-free condition $\partial_j(\tilde{G}\tilde{u}^j)=0$. 
We now only need to make sure that $\tilde{G}\tilde{u}^j \partial_j \tilde{u}^i\in C^0_{r,\theta}$ for all $i=1,\dots,d$.
Now observe that, thanks to \cref{l:regular_extension}, the functions $\partial_j\tilde{u}^{i}$ are continuous when $j\geq 2$ (or, equivalently, all the derivatives that are not in the $r$ direction), so that when $(i,j) \in \{1, \dots, d\} \times \{2, \dots, d\}$ the function $\tilde{G}\tilde{u}^j\partial_j\tilde{u}^i$ is continuous and has no jumps.
Finally, let us deal with the case $j=1$. Since $\tilde u^1$ is uniformly continuous with $\tilde u^1(0,\theta)=0$ and $\partial_1 \tilde{u}^i$ is bounded and also continuous on $\{r\neq 0\}$ (in the case $i=1$ even $\partial_1\tilde{u}^1$ is continuous by \cref{l:regular_extension}), it follows that the function $\tilde{G}\tilde{u}^1\partial_1\tilde{u}^i$ is continuous. 
Therefore there is no jump of $\tilde{G}\tilde{u}^1\partial_1\tilde{u}^i$ across $\{r=0\}$ for all $i=1,\dots,d$. This proves our claim. 
\end{proof}

\begin{remark} 
\cref{l:dd ext} uses the straightening of the boundary with geodesic coordinates in an essential way.
Indeed, the terms $\tilde{u}^1\partial_1\tilde{u}^i$ for 
 $i=1,\dots,d$ do not produce distributional jumps across $\{r=0\}$ because $\tilde u^1(0,\theta)=0$.
 Also the terms $\tilde{u}^j\partial_j\tilde{u}^i$ for $(i,j) \in \{1, \dots, d\} \times \{2, \dots, d\}$ are fine thanks to the additional regularity in the tangential direction given by \cref{l:regular_extension}.
\end{remark}

Combining the above lemmas (and keeping in mind that, in virtue of \cref{r:constant_neglect}, we can neglect the constant $A$) we infer that, on $U=(-r_0,r_0) \times \Theta$, the functions $\tilde{u}$ and $\tilde{p}$ satisfy  
\begin{equation}
    \label{eq.pressureExtended}
    -\partial_i\left(\tilde{g}^{ij}\tilde{G}\partial_j\tilde{p}\right) = \partial_{ij}^2\left(\tilde{G}\tilde{u}^i\tilde{u}^j\right)\quad \text{in } U.
\end{equation}

\subsection{Extension to the whole space}

Our goal is now to extend the solution $\tilde{p}$ and the right-hand side datum $\tilde{u}$ of~\eqref{eq.pressureExtended} to $\mathbb{R}^d$. In order to do so, let $\psi$ be a (non-negative, smooth) localization function  supported in  $(-r_0,r_0) \times \Theta$ and which equals $1$ on  $U_0\coloneqq(-r_0/2,r_0/2)\times \Theta /2$.

Let $\bar{p}\coloneqq\psi \tilde{p}$, $\bar{u} \coloneqq\psi \tilde{u}$ defined as functions on $\mathbb{R}^d$. Also introduce $\bar{G}=\psi \tilde{G}$ and $\bar{g}^{ij}=\psi\tilde{g}^{ij}$. Also, let  $\tilde{\psi}$ be another localization function such that $\operatorname{supp} \tilde \psi\subset \{\psi \equiv 1\}$ and a third one $\tilde{\tilde{\psi}}$ supported in $\{\tilde{\psi} \equiv 1\}$. In particular, observe that $\tilde{\psi} \partial_i \psi = 0$. Therefore, we can compute 
\[
    \tilde{\tilde{\psi}} \partial_i\left(\tilde{\psi}\bar{G}\bar{g}^{ij}\partial_j\bar{p}\right)=\tilde{\tilde{\psi}} \partial_i \left(\tilde{\psi}\tilde{G}\tilde{g}^{ij}(\partial_j\psi \tilde{p} + \psi\partial_j\tilde{p})\right),
\]
where we used that $\tilde{\psi} \psi = \tilde{\psi}$. Now, because $\tilde{\psi} \partial_j \psi = 0$, we obtain  
\[
     \tilde{\tilde{\psi}} \partial_i \left(\tilde{\psi}\bar{G}\bar{g}^{ij}\partial_j\bar{p}\right)=\tilde{\tilde{\psi}}\partial_i \left(\tilde{\psi}\tilde{G}\tilde{g}^{ij}\partial_j\tilde{p}\right),
\]
so that, finally, using $\tilde{\tilde{\psi}}\tilde{\psi}=\tilde{\tilde{\psi}}$ and $\tilde{\tilde{\psi}} \partial_j \tilde{\psi} = 0$, we obtain 
\[
    \tilde{\tilde{\psi}} \partial_i \left(\tilde{\psi}\bar{G}\bar{g}^{ij}\partial_j\bar{p}\right)=  \tilde{\tilde{\psi}} \partial_i \left(\tilde{G}\tilde{g}^{ij}\partial_j\tilde{p}\right).
\]
Similarly, one can verify that 
\[
    \tilde{\tilde{\psi}} \partial_{ij}^2 \left(\tilde{\psi}\bar{G}\bar{u}^i\bar{u}^j\right) = \tilde{\tilde{\psi}} \partial_{ij}^2 \left(\tilde{G}\tilde{u}^i\tilde{u}^j\right),
\]
so that, in conclusion, the equation satisfied by $\bar{u}$ and $\bar{p}$ is 
\begin{equation}
    \label{eq.pressureRd}
    -\tilde{\tilde{\psi}} \partial_i \left(\tilde{\psi}\bar{G}\bar{g}^{ij}\partial_j \bar{p}\right) = \tilde{\tilde{\psi}} \partial_{ij}^2 \left(\tilde{\psi}\bar{G}\bar{u}^i\bar{u}^j\right).
\end{equation}

\section{Boundary regularity, part 2: H\texorpdfstring{\"o}{ö}lder regularity in the full space}\label{s:main_proof_new}

From now on we will always consider $\partial \Omega\in C^{2,1}$, which guarantees that the new metric, as well as its even extension, is globally Lipschitz.
We will rewrite~\eqref{eq.pressureRd} as a pseudodifferential equation. First, let us change notation a little bit by writing $g^{ij}$ instead of $\tilde{\psi}\bar{G}\bar g^{ij}$, $ G$ instead of $\tilde \psi \bar G$, $q$ instead of $\bar{p}$, $\tilde{u}^i$ instead of $\bar{u}^i$ and $\psi$ instead of $\tilde{\tilde{\psi}}$.  
For sake of clarity, in our new notation, the equation is
\begin{equation}\label{p_boundary_full_space_new_notation}
   - \psi \partial_i\left(g^{ij}\partial_j q\right) =\psi \partial_{ij}^2 \left(G\tilde{u}^i\tilde{u}^j\right).
\end{equation}

We aim to prove our last result, namely, the $C^{2\gamma}_*$ estimate on \eqref{p_boundary_full_space_new_notation}.  More precisely, in this last section we shall prove that
\begin{equation}
    \label{est:p_boundary_Rd_new}
    \left\|  q\right\|_{C^{2\gamma}_*(\R^d)}\leq C \left( \|G\tilde u\|_{C^\gamma(\R^d)}^2+\left\|q\right\|_{L^\infty(\R^d)} \right).
\end{equation}
Note that, since the change of variables is bi-Lipschitz, the previous estimate for $q$ automatically translates into  \eqref{p_contin_est_boundary_local_noboundary} for a certain ball $B_{R_0}(x_0)$, since we have chosen the cutoff function $\psi$ such that $\psi\equiv 1$ in an open neighborhood of $x_0$.

\jump

The operator $E_{1,i}:=g^{ij}\partial_j$ is a pseudodifferential operator, $E_{1,i}=\operatorname{Op}(e_{1,i})$, where $e_{1,i}:=e_{1,i}(r,\theta,\xi)\in C^1_*S^1$ is the symbol defined by 
\[
e_{1,i}(r,\theta,\xi):=g^{ij}(r,\theta) \xi^j.
\]

Let $\delta \in (0,1)$. We define the \emph{sharp} part of $e_{1,i}$ as 
\begin{equation}
    \label{eq.sharpE2_new}
    e_{1,i}^{\sharp}(r,\theta,\xi):=\sum_{K\leq M^{\delta}} g_K^{ij}(r,\theta)\mathbf{P}_M(\xi)\xi^{j}, 
\end{equation}
where $g_K^{ij}:=\mathbf{P}_Kg^{ij}$ and the summation runs on all $M,K\in 2^{\mathbb{N}}$ such that $K \leq M^{\delta}$. We also define the \emph{flat} part of the symbol $e_{1,i}$ as
\[
    e_{1-\delta,i}^{\flat} \coloneqq e_{1,i} - e_{1,i}^{\sharp}.
\] 
Note that $e_{1,i}^{\sharp}$ is a symbol of order $1$. More precisely, we have the following

\begin{lemma} For all $i=1, \dots, d$, we have $e_{1,i}^{\sharp} \in S^1_{1,\delta}$.
\end{lemma}

\begin{proof} 
To prove that $e_{1,i}^{\sharp} \in S^1_{1,\delta}$, we need to show the convergence of the series in \eqref{eq.sharpE2_new} in all $C^k_x$ spaces (for any non-negative $k$), with the right bounds. Note that, for any $K$, $g_K^{ij}$ is a smooth function, and recall that $\mathbf{P}_M(\xi)$ is the Littlewood--Paley partition introduced in \cref{s:function_spaces}.  Let $\alpha,\beta$ be two multi-indices, with $|\beta|=k$, and set $x=(r,\theta)$. 
By the triangle inequality, we have
    \[
        |\partial _{\xi}^{\alpha}\partial_{x}^{\beta}e^{\sharp}_{1,i}(x,\xi)| \leq \sum_{K\leq M^{\delta}} \left|\partial_x^{\beta}g_K^{ij}(x)\right|\left|\partial^{\alpha}_{\xi}(\mathbf{P}_M(\xi)\xi^{j})\right|. 
    \]
    Now observe that, by \cref{thm.bernstein}, there holds  $|\partial_x^{\beta}g_K^{ij}(x)| \lesssim K^{|\beta|}\|g^{ij}\|_{L^{\infty}} \lesssim K^{|\beta|}$. 
    By direct computations, as soon as $M>1$, one also has the bound 
    \[
        |\partial^{\alpha}_{\xi}(\mathbf{P}_M(\xi)\xi^{j})| \lesssim \langle \xi \rangle ^{1-|\alpha|} \chi(M^{-1}\xi),
    \] 
    where $\chi$ is some compactly supported function in the annulus $\left\{1/4\leq |\xi|\leq 4 \right\}$. Therefore, 
    \[
        |\partial _{\xi}^{\alpha}\partial_{x}^{\beta}e^{\sharp}_{1,i}(x,\xi)| \lesssim \sum_{K \leq M^{\delta}} K^{|\beta|} \langle \xi \rangle ^{1-|\alpha|}\chi(M^{-1}\xi). 
    \]
    After a summation in $K$ and $M$, we hence get
    \[        
        |\partial _{\xi}^{\alpha}\partial_{x}^{\beta}e^{\sharp}_{1,i}(x,\xi)| \lesssim \sum_{M} M^{\delta |\beta|} \langle \xi \rangle ^{1-|\alpha|}\chi(M^{-1}\xi) \lesssim \langle \xi \rangle ^{1-|\alpha| + \delta |\beta|}
    \]
    and the conclusion follows.
\end{proof}

\begin{remark}\label{rem.flatPart}
With similar computations, one gets $e_{1-\delta,i}^{\flat} \in C^1_*S^{1-\delta}_{1,\delta}$, therefore justifying its indices. However, since it is of no use here, we leave the proof of this property to the reader.
\end{remark}

Thus, equation \eqref{p_boundary_full_space_new_notation} translates into 
\[
    -\psi \partial_i \left( E_{1,i}^\sharp (q)\right)=\psi \partial_{ij} \left(G \tilde u^i \tilde u^j \right)+\psi  \partial_i \left( E_{1-\delta,i}^\flat (q)\right).
\]
Our next observation is  that the principal part of the operator $\psi \partial_i \left( E_{1,i}^\sharp \, \cdot\,\right)$ is elliptic. In order to see it, we apply \cref{thm.calculus} with $N=0$, and get $\psi \partial_i \left( E_{1,i}^\sharp \, \cdot\,\right) = \operatorname{Op}(c_i)$ where
\[
    c_i(r,\theta,\xi)=\psi e^\sharp_{1,i}(r,\theta,\xi)\xi^i + h(r,\theta,\xi),
\]
and where $h\in S^{1+\delta}_{1,\delta}$. Since $1+\delta<2$, the operator $R_{1+\delta}:=\operatorname{Op}(h)$ can be considered as a remainder term. Thus, we can further rewrite the equation as 
\begin{equation}\label{eq_second_rewriting}
- E_{2}^\sharp q=\psi \partial_{ij} \left(G \tilde u^i \tilde u^j \right)+\psi  \partial_i \left( E_{1-\delta,i}^\flat q\right)+R_{1+\delta} (q),
\end{equation}
where $E_{2}^\sharp :=\operatorname{Op}(e^\sharp_2)$ is the second-order operator associated to the symbol
\[
    e^\sharp_2:=\psi(r,\theta) \sum_{K\leq M^{\delta}} g_K^{ij}(r,\theta)\mathbf{P}_M(\xi)\xi^{i}\xi^j.
\]


\begin{lemma}
The symbol $e^\sharp_2\in S^2_{1,\delta}$ is elliptic.
\end{lemma}

\begin{proof}
Since $e^\sharp_{1,i}\in S^1_{1,\delta}$ and $\xi^i\in S^1_{1,\delta}$, we immediately see that  $e^\sharp_2\in S^2_{1,\delta}$.
To prove the ellipticity of $e_2^{\sharp}$, let us first observe that the symbol 
\[
e_2(r,\theta,\xi):=\psi(r,\theta)g^{ij}(r,\theta)\xi^i\xi^j
\]
is elliptic.
Indeed, by \cref{prop.geometry}, we have $g^{ij}(x)\xi^{i}\xi^{j} \geq c |\xi|^2$, for all $|\xi| \geq R$ and $x \in U$, for some constant $c>0$. 
Since $\operatorname{supp} \psi \subset U$, the latter inequality holds for all $x$ in $U_0 \coloneqq \{\psi \equiv 1\}$, hence implying that $e_2$ is elliptic on $U_0$ with constant $c$.

In order to prove that $e_2^{\sharp}$ is elliptic, we only need to prove that, for $M_0$ large enough and $|\xi| \geq M_0$ there holds  
\begin{equation}
    \label{eq.largeM_new}
    \sum_{K\leq M^{\delta}} g_K^{ij}(x)\mathbf{P}_M(\xi)\xi^{i}\xi^{j} \geq \frac{c}{2}|\xi|^2,
\end{equation}
where the summation runs on both $K$ and $M$.

To obtain~\eqref{eq.largeM_new}, let us bound 
\begin{align*}
    \left\vert \sum_{K > M^{\delta}} g_K^{ij}(x)\mathbf{P}_M(\xi)\xi^{i}\xi^{j} \right\vert & \leq \sum_{K > M^{\delta}} \left|g_K^{ij}(x)\right|\Big|\mathbf{P}_M(\xi) \Big|\left|\xi^{i}\right|\left|\xi^{j}\right|  \leq d|\xi|^2 \sum_{K:K> M_0^{\delta}}\left\|g_K\right\|_{L^{\infty}},
\end{align*}
where we have also used the partition of unity property $\sum_{M} \mathbf{P}_M(\xi)=1$ to get rid of the sum on $M$. Now observe that, using \eqref{eq.normHolderZ}, there holds 
\[
   \sum_{K:K>M_0^{\delta}} \|g_K\|_{L^{\infty}} \leq C\sum_{K:K>M_0^{\delta}} K^{-1}\|g\|_{C^{1}_*} \leq 2C M_0^{-\delta }\|g\|_{C^{1}_*},
\]
which can be made smaller than $\frac{c}{2d}$ for $M_0$ large enough.  
Therefore, we get~\eqref{eq.largeM_new}. 
\end{proof}

Since $E_2^{\sharp}= \operatorname{Op}(e_2^{\sharp})$ is elliptic on $U_0=\{\psi \equiv 1\}$,  we can invert it by using \cref{thm.elliptic}.
Hence there exist $E_{-2}^{\sharp} \in \operatorname{Op}(S^{-2}_{1,\delta})$ and $R_{-2(1-\delta)} \in \operatorname{Op}(S^{-2(1-\delta)}_{1,\delta})$ such that 
\begin{equation}
    \label{eq.inversionEll_new}
    E^{\sharp}_{-2}\circ E^{\sharp}_{2} = \operatorname{Op}(\chi) + R_{-2(1-\delta)},
\end{equation}
where $\chi$ is some compactly supported function in $U_0$ such that $\chi \equiv 1$ in the open set $U_1$ containing the point $x_0$. 
First, since $2\gamma \leq 1$ and $\delta < \frac{1}{2}$, we have $2\gamma-2(1-\delta) < 0$ and therefore, thanks to \cref{th.calderonVaillancourt}, we have the continuity property for $R_{-2(1-\delta)}$, which reads as
\begin{equation}
    \label{eq.continuityRinfty_new}
    \|R_{-2(1-\delta)}(q)\|_{C^{2\gamma}_*} \lesssim \|q\|_{C^0_*} \lesssim\|q\|_{L^{\infty}}. 
\end{equation}

We are going to perform a last change of variable, by introducing $v^i \coloneqq G\tilde{u}^i$ and $a \coloneqq \frac{1}{G} \in C^1_*$, so that $\psi \partial_{ij}^2(G\tilde{u}^i\tilde{u}^j) = \psi \partial_{ij}^2(av^iv^j)$. Note that the divergence-free condition now reads as 
\begin{equation}
    \label{v_div-free_new}
    \partial_iv^i=0.
\end{equation}
Observe that~\eqref{est:p_boundary_Rd_new} can be rephrased as
\begin{equation}
    \label{est:finalBoudary_new}
    \| q\|_{C_*^{2\gamma}} \leq C\left(\|v\|_{C^{\gamma}_*}^2 + \|q\|_{L^{\infty}}\right). 
\end{equation}
In order to obtain~\eqref{est:finalBoudary_new}, let us apply $-E_{-2}^{\sharp}$ to~\eqref{eq_second_rewriting}.
Hence we can write 

\begin{equation*}
    \chi q = -\underbrace{E_{-2}^{\sharp}\left(\psi \partial^2_{ij}(av^iv^j)\right)}_{A} - \underbrace{E_{-2}^{\sharp} \left(\psi  \partial_i \left( E_{1-\delta,i}^\flat q\right)\right)}_{B} - \underbrace{ E_{-2}^{\sharp}\circ R_{1+\delta}(q)+R_{-2(1-\delta)}(q)}_{R},
\end{equation*}
so we just need to estimate the right-hand side. 
The term $A$ is the main contribution, whereas the terms $B$ and $R$ should be thought as remainder terms. We postpone the treatment of $A$ and $B$ to \cref{sec:termA_new} and \cref{sec:termB_new}.

\jump

Let us consider the term $R$. Using the continuity property $C^{2\gamma -2}_* \to C^{2\gamma}_*$ of $E_{-2}^{\sharp} \in \operatorname{Op}(S^{-2}_{1,\delta})$ given by \cref{th.calderonVaillancourt}, together with the continuity property \eqref{eq.continuityRinfty_new}, we can estimate
\begin{align*}
    \|R\|_{C^{2\gamma}_*} &\lesssim \|R_{1+\delta}(q)\|_{C^{2\gamma -2}_*}+ \|R_{-2(1-\delta)}(q)\|_{C^{2\gamma}_*} \\
    & \lesssim \|q\|_{C^{2\gamma - 1 +\delta}_*} + \|q\|_{L^\infty} \lesssim  \|q\|_{C_*^{2\gamma-\frac{\delta}{2}} } ,
\end{align*}
where in the third inequality we have used that $2\gamma  \leq 1$ and that $4\gamma>3\delta$ (this last property can be indeed assumed without any loss of generality, by choosing $\delta>0$ sufficiently small). Now, by interpolation between $C^0_*$ and $C^{2\gamma}_*$ and Young's inequality, we can estimate 
\begin{align*}
 C\|q\|_{C_*^{2\gamma -\frac{\delta}{2}}} &  \leq C\|q\|_{C_*^{2\gamma}}^{1-\frac{\delta}{4\gamma}}\|q\|_{L^{\infty}}^{\frac{\delta}{4\gamma}}  \leq \varepsilon \|q\|_{C^{2\gamma}_*} + C(\varepsilon) \|q\|_{L^{\infty}},
\end{align*}
for some (possibly large) constant $C(\eps)>0$.
In the following subsections we will prove that
\begin{equation}
    \label{est.A_new}
    \|A\|_{C^{2\gamma}_*} \lesssim \|v\|_{C^{\gamma}}^2
\end{equation}
and 
\begin{equation}
    \label{est.B_new}
    \|B\|_{C^{2\gamma}_*} \lesssim \|q\|_{C_*^{2\gamma - \frac{\delta}{2}}} \leq \varepsilon\|q\|_{C^{2\gamma}_*} + C(\varepsilon)\|q\|_{L^{\infty}}. 
\end{equation}

The combination of these inequalities indeed imply a bound of the form  
\[
    \|\chi q\|_{C^{2\gamma}_*} \leq 2\varepsilon \|q\|_{C^{2\gamma}_*} + C(\varepsilon)(\|q\|_{L^{\infty}} + \|v\|^2_{C^{\gamma}_*})
\]
(for all $\varepsilon > 0$ arbitrarily small),
which itself implies~\eqref{est:finalBoudary_new}.
Indeed, the estimate holds in $\{\chi \equiv 1\}$, and can be obtained (with the same proof) locally around any point in $U_0$.
Thus the term $\|q\|_{C^{2\gamma}_*}$ in the right-hand side can be absorbed in the left-hand side by using standard techniques in elliptic PDEs, see~\cite{FR23}*{Proof of Theorem 2.16} or~\cite{GT}*{Theorem 6.2} for instance.

\jump

The goal of the following subsections is to prove the two last estimates \eqref{est.A_new} and \eqref{est.B_new}.

\subsection{Term A}\label{sec:termA_new}
 Let us write 
\[
    a= \sum_{L \in 2^{\mathbb{N}}} a_L \quad \text{ and } \quad v^i = \sum_{K \in 2^{\mathbb{N}}} v_K^i,
\]
so that 
\[
    A = E_{-2}^{\sharp}\left(\sum_{L, K, M \geq 1} \psi \partial_{ij}^2\left(a_Lv_K^iv_M^j\right)\right). 
\]

At this stage, our goal is to apply the same strategy of \cref{s:interior}. To do so, we use the double-divergence form when the frequencies of $v$ are similar, $K\sim M$.
This allows us to split the derivatives and gain the double regularity. Also, when $L \leq \max\{K,M\}$, the derivatives $\partial_{ij}^2$ apply to $a$, which is smoother than $v$, therefore the double-divergence form also gives us the double regularity. 
When $L$ is not the largest frequency and one of the frequencies among $K, M$ is significantly larger than the other one, the double-divergence form could not achieve double regularity. Therefore, in this case, we rely on the divergence-free structure of $v$, which allows to rewrite the double-divergence as some product of terms of order $1$. 
Indeed, we have
\begin{lemma}\label{lem.rewriting_new} There holds
\[
    \partial_{ij}^2(a_Lv^i_Kv^j_M)=\partial_j(\partial_i a_L v^i_Kv^j_M) + \partial_j a_L v^i_K\partial_i v^j_M + a_L \partial_jv^i_K\partial_i v^j_M.
\]
\end{lemma}

\begin{proof} Note that $\mathbf{P}_K$ commutes with $\partial_i$, as they are both Fourier multipliers.
Therefore, since  $\partial_i v^i=0$ by \eqref{v_div-free_new}, we obtain $\partial_iv^i_K=0$. Using this, we can write 
\begin{align*}
     \partial_{ij}^2(a_Lv_K^iv_M^j) &=\partial_j(\partial_i a_L v_K^iv_M^j) + \partial_j(a_Lv^i_K\partial_iv^j_M) \\
     & = \partial_j(\partial_i a_L v_K^iv_M^j) + \partial_ja_Lv^i_K\partial_iv^j_M + a_L\partial_jv^i_K\partial_i v^j_M,
\end{align*}
where in the last line we used again $\partial_jv^j_M=0$.
\end{proof}

Thanks to the above discussion, we write $A=A_1 + A_2$, with
\begin{equation*}
    \label{eq.A1_new}
    A_1\coloneqq E^{\sharp}_{-2}\left(\sum_{\substack{K\sim M\\L \ll \text{max}\{K,M\}}}+ \sum_{L\geq \frac{1}{8}\text{max}\{K,M\}}\right)\psi \partial^2_{ij}(a_Lv_K^iv_M^j) =: A_{11}+A_{12}
\end{equation*}
and 
\begin{align*}
    \label{eq.A2_new}
    A_2& \coloneqq E^{\sharp}_{-2}\sum_{\substack{K\ll M \text{ or } M\ll K\\L \ll \text{max}\{K,M\}}}\psi \partial^2_{ij}(a_Lv_K^iv_M^j) \notag \\ 
    & =  E^{\sharp}_{-2}\psi \sum_{\substack{K\ll M \text{ or } M\ll K\\L \ll \text{max}\{K,M\}}} \left(\partial_j(\partial_i a_L v^i_Kv^j_M) +\partial_j a_L v^i_K\partial_i v^j_M + a_L \partial_jv^i_K\partial_i v^j_M \right) \\
    &=: A_{21} + A_{22} + A_{23} , \notag 
\end{align*}
where we also  used \cref{lem.rewriting_new}. We are left with estimating $A_{11}$, $A_{12}$, $A_{21}$, $A_{22}$ and $A_{23}$.

\jump

As a preliminary remark, let us observe that, by composition of pseudodifferential operators with smooth symbols (see \cref{thm.calculus}), we recognize that $E_{-2}^{\sharp}\psi\partial_{ij}^2 \in \operatorname{Op}(S^0_{1,\delta})$, therefore continuous $C^{2\gamma}_* \to C^{2\gamma}_*$, thanks to \cref{th.calderonVaillancourt}. 

\medskip

\noindent \textbf{Estimating $A_{11}$.} 
By continuity of $E_{-2}^{\sharp}\psi\partial_{ij}^2$, we start by estimating
\[
    \|A_{11}\|_{C_*^{2\gamma}} \lesssim  \left\|\sum_{\substack{K\sim M\\L \ll \text{max}\{K,M\}}}a_Lv^i_Kv^j_M\right\|_{C^{2\gamma}_*}.
\]
As the sum is symmetrical in $i$ and $j$, we can assume $M\geq K$, therefore $K  \in \{\frac{M}{2},M\}$. As $v_M^j$ and $v_K^i$ are frequency supported at similar values $M \sim K$,  $v^i_Kv^j_M$ is frequency supported in $\{|\xi| \lesssim M\}$.
Since $L<M$, then $a_Lv^i_Kv^j_M$ is also frequency supported in $\{|\xi| \lesssim M\}$. Therefore, 
\[
    \mathbf{P}_N\left(\sum_{\substack{K\sim M\\L \ll \text{max}\{K,M\}}}a_Lv^i_Kv^j_M\right) = \sum_{\substack{K\sim M \geq N\\L \ll M}} \mathbf{P}_N(a_Lv^i_Kv^j_M),
\]
so that 
\[
    \|A_{11}\|_{C_*^{2\gamma}} \lesssim \sup_{N \geq 1} N^{2\gamma} \sum_{\substack{M \geq N \\ L\ll M}}\sum_{K=\frac{M}{2},M}  \left\|a_Lv^i_Kv^j_M\right\|_{L^{\infty}}. 
\]
Using $a \in C^1_*$ and $v \in C^{\gamma}_*$, we see that
\[
     \left\|a_Lv^i_Kv^j_M\right\|_{L^{\infty}} \lesssim L^{-1}(KM)^{-\gamma}\|v\|^2_{C^{\gamma}_*}\|a\|_{C^1_*} \lesssim L^{-1}M^{-2\gamma}\|v\|_{C^{\gamma}_*}^2,
\] 
where we used that $K \sim M$. Therefore, we finally arrive at
\begin{align*}
    \|A_{11}\|_{C_*^{2\gamma}} &\lesssim \sup_{N \geq 1} \sum_{\substack{M \geq N \\ L\ll M}} N^{2\gamma}L^{-1}M^{-2\gamma}\|v\|^{2}_{C_*^{\gamma}}  \lesssim \sup_{N \geq 1} \sum_{\substack{M \geq N \\ L \geq 1}} N^{2\gamma}L^{-1}M^{-2\gamma}\|v\|^{2}_{C_*^{\gamma}} \\
    & \lesssim \sup_{N \geq 1} N^{2\gamma} \sum_{M \geq N}M^{-2\gamma}\|v\|^{2}_{C_*^{\gamma}} \lesssim \|v\|^{2}_{C_*^{\gamma}},
\end{align*}
proving that $A_{11}$ satisfies~\eqref{est.A_new}. 
\medskip 

\noindent \textbf{Estimating $A_{12}$.}  By continuity of $E_{-2}^{\sharp}\psi\partial_{ij}^2$, we start by estimating 
\[
    \|A_{12}\|_{C_*^{2\gamma}} \lesssim  \left\|\sum_{L\geq \frac{1}{8}\text{max}\{K,M\}} a_Lv^i_Kv^j_M\right\|_{C^{2\gamma}_*}.
\]
By symmetry of $K$ and $M$, we can assume $M \geq K$. 
The frequency localization gives
\begin{align*}
    \mathbf{P}_N\left(\sum_{L\geq \frac{1}{8}\text{max}\{K,M\}}a_Lv^i_Kv^j_M\right) &= \mathbf{P}_N\left(\sum_{L\gg M \geq K}a_Lv^i_Kv^j_M\right) + \mathbf{P}_N\left(\sum_{L\sim M \geq K}a_Lv^i_Kv^j_M\right) \\
    & = \sum_{\substack{L \sim N \\ K\leq M \ll L}}\mathbf{P}_N(a_Lv^i_Kv^j_M) + \sum_{\substack{N \leq L \sim M \\ K\leq M}}\mathbf{P}_N(a_Lv^i_Kv^j_M), 
\end{align*}
so that, being $\|\partial_j a_{L}\|_{L^\infty} \lesssim \|\partial_j a\|_{L^{\infty}} \lesssim 1$ since $a\in \Lip$  (see \cref{thm.bernstein}),  we deduce
\begin{align*}
    \|A_{12}\|_{C_*^{2\gamma}} &\lesssim \sup_{N \geq 1} N^{2\gamma} \left(\sum_{\substack{L \sim N \\ K\leq M \ll L}} L^{-1}M^{-\gamma}K^{-\gamma} + \sum_{\substack{N \leq L \sim M \\ K\leq M}} L^{-1}{M^{-\gamma}}K^{-\gamma} \right)\|v\|^2_{C^{\gamma}_*}\\
    & \lesssim \sup_{N \geq 1} N^{2\gamma} \left( N^{-1}\sum_{M\geq 1}M^{-\gamma}\sum_{K \geq 1}K^{-\gamma}  + \sum_{L\geq N} L^{-(1+\gamma)} \sum_{K\geq 1}K^{-\gamma}\right)\|v\|^2_{C^{\gamma}_*}\\
    & \lesssim \sup_{N \geq 1} \left(N^{2\gamma -1} + N^{\gamma -1}\right)\|v\|^2_{C^{\gamma}_*} \lesssim \|v\|^2_{C^{\gamma}_*},
\end{align*}
since $2\gamma -1 \leq 0$. This proves that $A_{12}$ satisfies~\eqref{est.A_new}. 

\medskip

\noindent \textbf{Estimating $A_{21}$.} To estimate $A_{21}$, we use the  $C^{2\gamma -1}_* \to C_*^{2\gamma}$ continuity of $E_{-2}^{\sharp}\psi\partial_{j} \in \operatorname{Op}(S^{-1}_{1,\delta})$.
Since $K$ and $M$ play a symmetrical role, let us assume that $K \ll M$ (which means $K<\frac{M}{2}$), so
\[
    \|A_{21}\|_{C^{2\gamma}_{*}} \lesssim \left\| \sum_{\substack{K \ll M \\ L \ll M}} \partial_ia_L v^i_Kv^j_M \right\|_{C^{2\gamma -1}_*}. 
\]
Now observe that 
\[
    \mathbf{P}_N\left(\sum_{\substack{K \ll M \\ L \ll M}} \partial_ia_L v^i_Kv^j_M \right) = \sum_{\substack{M\sim N \\ K, L \ll M}} \mathbf{P}_N(\partial_ia_Lv^i_Kv^j_M),
\]
and therefore 
\begin{align*}
    \|A_{21}\|_{C_*^{2\gamma}} &\lesssim \sup_{N\geq 1}N^{2\gamma -1} \sum_{\substack{M\sim N \\ K, L \ll M}}  \left\|\partial_ia_Lv^i_Kv^j_M\right\|_{L^{\infty}}.
\end{align*}
Since $a\in \Lip$, we have $\|\partial_j a_{L}\|_{L^\infty} \lesssim \|\partial_j a\|_{L^{\infty}} \lesssim 1$, see \cref{thm.bernstein}.
Since $v\in C_*^{\gamma}$, there holds 
\[
     \left\|\partial_ia_Lv^i_Kv^j_M\right\|_{L^{\infty}} \lesssim (KM)^{-\gamma}\|v\|_{C^{\gamma}_*}^2,
\]
and we can bound 
\begin{align*}
    \|A_{21}\|_{C_*^{2\gamma}} &\lesssim \sup_{N\geq 1}N^{2\gamma -1}\|v\|_{C^{\gamma}_*}^2 \sum_{\substack{M\sim N \\ K, L \ll M}} (KM)^{-\gamma}\lesssim \sup_{N\geq 1}N^{\gamma -1}\log(N)\|v\|_{C^{\gamma}_*}^2 \lesssim \|v\|_{C^{\gamma}}^2,
\end{align*}
because $\gamma - 1 < 0$.
\medskip

\noindent \textbf{Estimating $A_{22}$.} To estimate $A_{22}$, we use the  $C^{2\gamma -2}_* \to C_*^{2\gamma}$ continuity of $E_{-2}^{\sharp}\psi\in \operatorname{Op}(S_{1,\delta}^{-2})$,  together with $\|\partial_j a_{L}\|_{L^\infty} \lesssim \|\partial_j a\|_{L^{\infty}} \lesssim 1$ (see \cref{thm.bernstein}), so
\begin{align*}
    \|A_{22}\|_{C^{2\gamma}_{*}} &\lesssim \left\| \sum_{\substack{K \ll M \\ L \ll M}} \partial_ja_L v^i_K\partial_i v^j_M \right\|_{C^{2\gamma -2}_*} + \left\| \sum_{\substack{M \ll K \\ L \ll K}} \partial_ja_L v^i_K\partial_i v^j_M \right\|_{C^{2\gamma -2}_*}. 
\end{align*}
For the first term, we can write 
\begin{align*}
    \left\| \sum_{\substack{K \ll M \\ L \ll M}} \partial_ja_L v^i_K\partial_i v^j_M \right\|_{C^{2\gamma -2}_*} & \lesssim \sup_{N\geq 1} N^{2\gamma -2} \|v\|^2_{C_*^{\gamma}}\sum_{K,L\ll M \sim N} K^{-\gamma}M^{1-\gamma}\\
    & \lesssim \sup_{N\geq 1}N^{\gamma -1}\log(N) \|v\|^2_{C_*^{\gamma}} \lesssim \|v\|^2_{C_*^{\gamma}},
\end{align*}
because $\gamma - 1< 0$. To estimate the second term, we observe that
\begin{align*}
    \left\| \sum_{\substack{M \ll K \\ L \ll K}} \partial_ja_L v^i_K\partial_i v^j_M \right\|_{C^{2\gamma -2}_*} 
    & \lesssim \sup_{N\geq 1}N^{2\gamma -2}\|v\|^2_{C_*^{\gamma}} \sum_{  \substack{M \ll K\sim N \\ L \ll K} } \|\partial_ja_{L}\|_{L^{\infty}} K^{-\gamma}M^{1-\gamma} \\
    & \lesssim\sup_{N\geq 1}N^{-1}\log(N)\|v\|^2_{C_*^{\gamma}} \lesssim \|v\|^2_{C_*^{\gamma}}.
\end{align*}

\medskip

\noindent \textbf{Estimating $A_{23}$.} 
We start by using the  $C^{2\gamma-2}_* \to C_*^{2\gamma}$ continuity of $E_{-2}^{\sharp}\in \operatorname{Op}(S_{1,\delta}^{-2})$ (see \cref{th.calderonVaillancourt}).
Since again $K$ and $M$ play a symmetrical role, we also assume that $K\leq M$.
Hence, we can write
\[
    \|A_{23}\|_{C_*^{2\gamma}} \lesssim \left\|\sum_{\substack{K\ll M \\ L\ll M}} a_L\partial_jv^i_K\partial_iv^j_M\right\|_{C^{2\gamma -2}_*},
\]
Therefore, by taking into account the frequency localization, this gives
\begin{align*}
    \|A_{23}\|_{C_*^{2\gamma}} &\lesssim \sup_{N \geq 1} N^{2\gamma -2} \|v\|^2_{C^{\gamma}_*} \sum_{\substack{K \ll M \sim N\\ L \ll M}} L^{-1}K^{1-\gamma}M^{1-\gamma} \lesssim \|v\|^2_{C^{\gamma}_*},
\end{align*}
where we used that $1-\gamma > 0$. 

\subsection{Term B}\label{sec:termB_new} 
We start by estimating
\[
    \left\|E^\sharp_{-2}\left(\psi\partial_i\left( E^\flat_{1-\delta,i}(q)\right) \right)\right\|_{C^{2\gamma}_*}\lesssim \sum_{i=1}^d\left\|E^\flat_{1-\delta,i}(q) \right\|_{C^{2\gamma-1}_*}.
\]
By definition of $E^{\flat}_{1-\delta,i}$, we can write

\begin{align*}
    E_{1-\delta,i}^{\flat}(q) &= \sum_{K > M^{\delta}} g_{K}^{ij}\partial_{j} q_M =  \sum_{M^{\delta} < K \ll M} g_K^{ij} \partial_{j} q_M  +  \sum_{K \sim M} g_K^{ij} \partial_{j} q_M  +  \sum_{K \gg M} g_K^{ij} \partial_{j} q_M, 
\end{align*}
so that, by frequency localization of the above terms, there holds 
\[
    \mathbf{P}_N E_{1-\delta,i}^{\flat}(q)  =\sum_{\substack{M^{\delta} < K \ll M \\ M\sim N}} \mathbf{P}_N(g_{K}^{ij}\partial_{j} q_M) + \sum_{K \sim M \geq N} \mathbf{P}_N(g_{K}^{ij}\partial_{j} q_M) + \sum_{M \ll K \sim N} \mathbf{P}_N(g_{K}^{ij}\partial_{j} q_M).
\]
Thus, by using $\|q_M\|_{L^{\infty}} \lesssim M^{-2\gamma + \frac{\delta}{2}}\|q\|_{C_*^{2\gamma - \frac{\delta}{2}}}$ and $\|g^{ij}_K\|_{L^\infty}\lesssim K^{-1}$ (recall that the extended metric is Lipschitz), we obtain
\begin{align*}
    \|E_{1-\delta,i}^{\flat}(q)\|_{C_*^{2\gamma -1}} &\lesssim \sup_{N\geq 1} N^{2\gamma -1}\|q\|_{C_*^{2\gamma - \frac{\delta}{2}}} \left(\sum_{\substack{M^{\delta} < K \ll M \\ M\sim N}} + \sum_{\substack{K \sim M \geq N}} +  \sum_{\substack{M \ll K \sim N}} \right) K^{-1}M^{1-2\gamma+\frac{\delta}{2}} \\
    & \lesssim \sup_{N\geq 1} \left(N^{- \frac{\delta}{2}}  + N^{-1 + \frac{\delta}{2}} + N^{-1 + \frac{\delta}{2}}\right)\|q\|_{C_*^{2\gamma - \frac{\delta}{2}}} \lesssim \|q\|_{C_*^{2\gamma - \frac{\delta}{2}}}.
\end{align*}
Here we used $\displaystyle\sum_{M^{\delta} < K \ll M} K^{-1} \leq \sum_{K > M^{\delta}} K^{-1} \lesssim M^{-\delta}$ and, in the second summation, we also assumed $-2\gamma + \frac{\delta}{2} < 0$ (this can be clearly ensured by choosing $\delta>0$ sufficiently small).

\begin{remark} 
As mentioned in \cref{rem.flatPart}, $E^{\flat}_{1-\delta} \in \operatorname{Op}(C^1_*S^{1-\delta}_{1,\delta})$ by similar computations.
Therefore, the estimate of the term $B$ can be handled using a generalization of \cref{th.calderonVaillancourt} to non-smooth symbols, see~\cite{taylorIII}*{Chapter 13, Proposition 9.10}. 
However, we preferred to provide a proof adapted to our operator in order to keep the proof as self-contained as possible. 
\end{remark}

\section{Final comments and extensions}\label{rem.extensionsThm}

Let us now conclude our paper  by comparing it to the approach used in \cites{BT21,BBT23},  and also making some comments on possible extensions of our results.

\subsection{Comparison with \cite{BT21,BBT23}}\label{S:versus_BT} Let us consider the regularized equation \eqref{Neuman_problem_p_approx}.  In order to define a trace of the limiting (as $\varepsilon \rightarrow 0$) $\partial_n p$ at the boundary, the authors of \cites{BT21,BBT23} consider the modified pressure $P_\eps:=p^\eps + (u_\eps\cdot n)^2$ solving the problem 
\begin{equation}\label{BT_equation}
    \left\{
    \begin{array}{rcll}
        -\Delta P^\eps &=& \diver \diver (u_\eps\otimes u_\eps) + \Delta (u_\eps\cdot n)^2& \text{ in }\Omega\\[1ex]
        \partial_n P^\eps &=&u_\eps\otimes u_\eps : \nabla n & \text{ on } \partial \Omega. 
    \end{array}
    \right.
\end{equation}
They prove the uniform bounds $\|P^\eps\|_{C^\gamma(\Omega)}\leq C$ and $\|P^\eps\|_{H^{-2}(\partial \Omega)}\leq C$,   which give the existence of a uniform limit $P$ (up to subsequences), together with  the fact that such a limit has a well-defined $\partial_n P\in H^{-2}(\partial \Omega)$.  Since the sequence $(u_\eps \cdot n)^2$ stays bounded in $C^\gamma(\Omega)$, then it must hold $\|p^\eps\|_{C^\gamma(\Omega)}\leq C$ as well, from which we can find a uniform limit, say $\tilde p\in C^\gamma(\Omega)$.  By testing \eqref{BT_equation} with some $\varphi\in C^2(\overline \Omega)$, and since $u_\eps$ is $C^1(\overline \Omega)$ for all $\eps>0$,  we deduce that the couple $(u_\eps,p_\eps)$ solves \eqref{p_weaksol}. Passing to the limit, we deduce that also $(u,\tilde p)$ solves \eqref{p_weaksol}.  Since solutions to \eqref{p_weaksol} are unique up to constants, our \cref{thm.main} implies that $\tilde p=p\in C_*^{2\gamma}(\Omega)$, thus doubling the regularity of the pressure.  Clearly, since in general $(u\cdot n)^2$ is only $\gamma$-H\"older, double regularity does not transfer to $P$. 
However, our approach does not give any meaning to $\partial_n p\big|_{\partial \Omega}$, which is coherent with the fact that there exists a divergence-free vector field $u\in C^\gamma (\Omega)$, $\gamma<\frac12$, such that $\partial_n (u\cdot n)^2\not\in \mathcal D' (\partial \Omega)$, as it has been proved in \cite{BBT23}*{Section 8}. It is in fact one of the main features of~\cite{BT21,BBT23} to provide a meaning of $\partial_n p\big|_{\partial \Omega}$, and this shows the necessity of introducing a different boundary condition for the Neumann pressure problem to hope to interpret the boundary datum in a trace sense. However, our approach is to work directly with the formulation  \eqref{p_weaksol}, which does not requite $\partial_n p\big|_{\partial \Omega}$ to be well-defined. As already discussed in \cref{S:weaksol}, the formulation \eqref{p_weaksol} is indeed the natural one corresponding to the weak formulation of \eqref{E}.

\subsection{Rougher domains}
Essentially, the content of \cref{thm.main} is  the validity of the double regularity estimate~\eqref{p_contin_est} on $C^{2,1}$ domains. 
The pseudodifferential approach used in this article might achieve the same regularity on less regular domains. 
In the following discussion, we will forget about the fact that the local normal coordinate system used at the boundary requires $\partial \Omega\in C^{2,1}$. 

Let $\Omega$ be a $C^{1,\alpha}$ domain. Since $u \otimes u \in C^{\gamma}$ and $\nabla n \in C^{\alpha - 1}$, it is enough that $\gamma + \alpha - 1> 0$ to ensure the well-definedness of the boundary condition as a distribution (see \cref{lem.holderProduct}). Thus, in order to have $\partial_n p\in C^{2\gamma - 1}$, together with its compatibility with the boundary condition, we also need $2\gamma - 1 \leq \min\{\alpha - 1, \gamma\} = \alpha - 1$. The combination of these two conditions give the conjecture for the optimal domain regularity, see \cref{optimal_figure}. 

\begin{figure}
\includegraphics[width=0.5\textwidth]{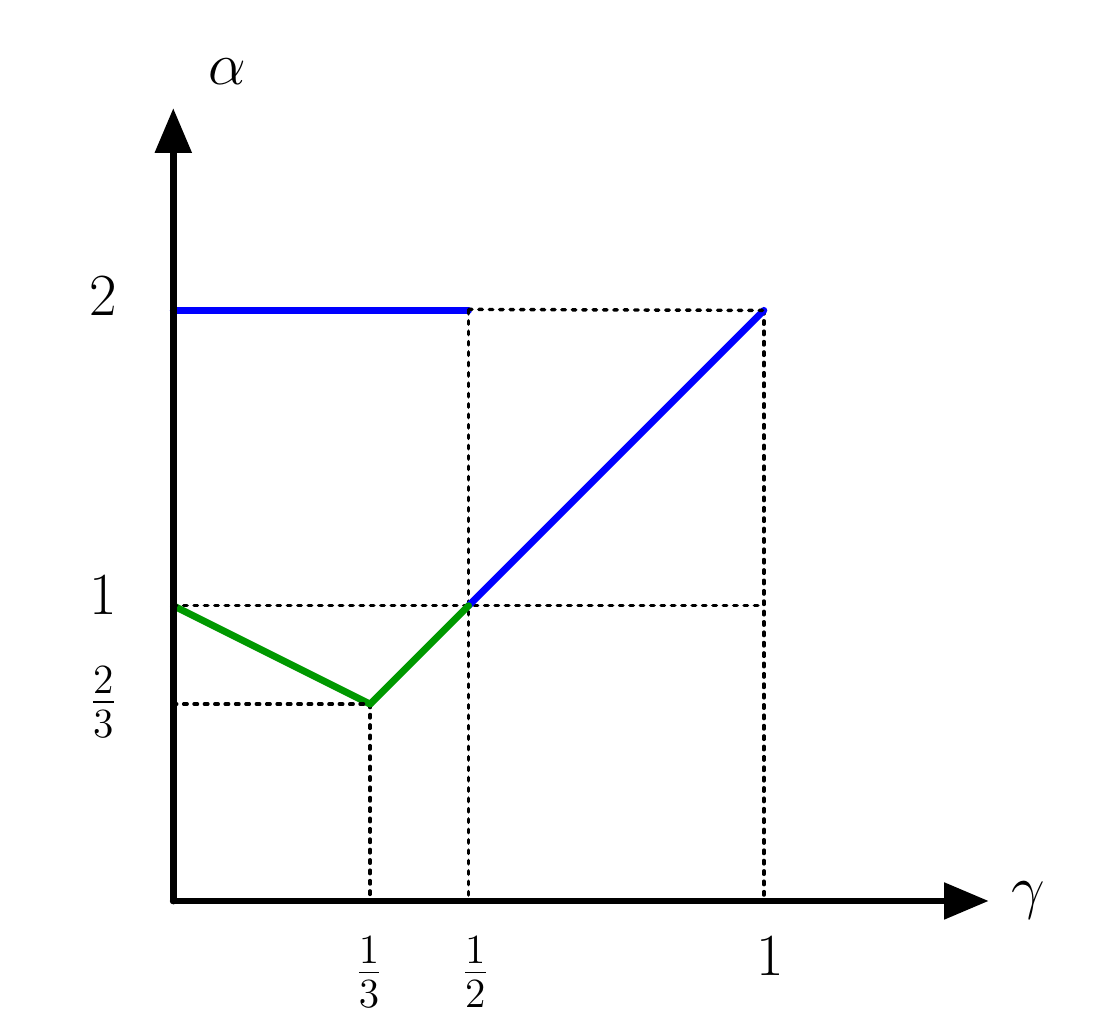}
\caption{Double H\"older regularity on $C^{1,\alpha}$ domains: blue lines represent what is proven in \cref{thm.main} for $\gamma \leq \frac{1}{2}$, as well as in~\cite{DLS22} when $\gamma > \frac{1}{2}$. Green lines represent the conjectured optimal domain regularity. Here we used the convention $C^{1,\alpha}=C^{2,\alpha-1}$, if $\alpha\in (1,2]$.}
\label{optimal_figure}
\end{figure}

We believe that the pseudodifferential part of the proof should still apply to this situation. 
Indeed, the regularity of $a \in C^{\alpha}$ in \cref{sec:termA_new} and in \cref{sec:termB_new} is enough for the argument. 
Thus, the main difficulty would be to extend the good approximation procedure of \cref{l:approx}.  
 
\subsection{Extension to Besov velocities} 
Another interesting extension of \cref{thm.main} would be to investigate the double regularity in Besov space, \textit{i.e.}, the validity of the estimate
\begin{equation}
\label{eq.besov}
\|p\|_{B^{2\gamma}_{r,\infty}} \lesssim \|u\|^2_{B^{\gamma}_{2r,\infty}},   
\end{equation}
for some values of $r\in[1,\infty)$, as $r=\infty$ is exactly the case of \cref{thm.main}. 
Indeed, from the pioneering work of A.~Kolmogorov \cite{K41}, the Besov classes seem the right setting in which to embed the local structure of turbulent flows. The estimate \eqref{eq.besov} has been indeed recently proved in \cite{CDF20} in the absence of the boundary, \textit{i.e.}, on $\T^d$ and on $\R^d$.
    
First, let us mention that, as in the previous observation, the pseudodifferential part will work similarly, with the slight difference that, instead of using continuity estimates of pseudodifferential operators in $C^s_*$ spaces, one should use the continuity estimates in Besov spaces, \textit{i.e.}, \cref{th.calderonVaillancourt} in Besov classes, which follows by interpolation between $L^r$  and $W^{s,r}$. 

Discarding the difficulty of approximation results such as \cref{l:approx}, a new problem arises, that is,  the traces on the boundary. Indeed, at a formal level, the trace of $u\in B^{\gamma}_{2r,\infty}$ at the boundary is only in $B^{\gamma - \frac{1}{2r}}_{2r,\infty}$.
Therefore, for smooth enough domains (say $C^{3}$, for example), and for $\gamma \geq \frac{1}{2r}$ (to let $u\otimes u$ be well defined), there holds $u \otimes u : \nabla n \in B^{\gamma - \frac{1}{2r}}_{r,\infty}$.
This is in fact compatible with $\partial_n p \in B^{2\gamma - 1 - \frac{1}{r}}_{r,\infty}$ when $\gamma \leq 1 + \frac{1}{2r}$, which is the case. 
So, for example, in the case $r=3$, which is the relevant case in the K41 theory for fully developed turbulence (see \cite{F95} for an extensive description), one should be able to obtain~\eqref{eq.besov} for all $\gamma \geq \frac{1}{6}$, thus including $\gamma=\frac13$.
This is in fact the exponent which plays a pivotal  role in turbulence theory relating to the Kolmogorov $\frac{4}{5}$-law: an exact result relating the energy dissipation of a solution to its third order (signed) structure function.

\color{black}

\end{document}